\documentclass[sn-mathphys-num]{sn-jnl}

\usepackage{graphicx}%
\usepackage{multirow}%
\usepackage{amsmath,amssymb,amsfonts}%
\usepackage{amsthm}%
\usepackage{mathrsfs,bbm}%
\usepackage[title]{appendix}%
\usepackage{parskip}
\usepackage{xcolor}%
\usepackage{textcomp}%
\usepackage{manyfoot}%
\usepackage{booktabs}%
\usepackage{algorithm}%
\usepackage{algorithmicx}%
\usepackage{algpseudocode}%
\usepackage{listings}%

\theoremstyle{thmstyleone}%
\newtheorem{theorem}{Theorem}
\theoremstyle{thmstyletwo}%
\newtheorem{remark}{Remark}%

\theoremstyle{thmstylethree}%
\newtheorem{definition}{Definition}%
\newtheorem{lemma}{Lemma}
\newcommand{\lin}{\mathrm{ lin}}
\newcommand{\sinc}{\mathrm{sinc}}
\begin{document}

\title{On the Convergence of Max-product and Max-Min Durrmeyer-type Exponential Sampling Operators}


\author[1]{Satyaranjan Pradhan}
\author[2]{Abhishek Senapati}
\author[1,*]{Madan Mohan Soren}

\affil[1]{Department of Mathematics, Berhampur University, 
Bhanja Bihar, Berhampur, 760007, Odisha, India}
\affil[2]{S.V.M (Autonomous) College, Jagatsinghpur, 754103, Odisha, India}

\date{}

\renewcommand{\thefootnote}{\fnsymbol{footnote}}
\footnotetext[1]{Satyaranjan Pradhan (\texttt{satya.math1993@gmail.com}); 
ORCID: \url{https://orcid.org/0009-0007-4207-3338}}
\footnotetext[2]{Abhishek Senapati (\texttt{abhisheksenapati108@gmail.com}); }
\footnotetext[3]{*Corresponding author:Madan Mohan Soren (\texttt{mms.math@buodisha.edu.in}); 
ORCID: \url{https://orcid.org/0000-0001-7344-8082}}

\abstract{This article discusses the convergence properties of the Max-Product and Max–Min variants of Durrmeyer-type exponential sampling series. We first establish pointwise and uniform convergence of both operators in the space of log-uniformly continuous and bounded functions. The rates of convergence are then analyzed in terms of the logarithmic modulus of continuity. Additionally, the approximation errors of the proposed operators are examined using a variety of kernel functions. Finally, graphical illustrations are provided to demonstrate the convergence behavior of both operators.}



\maketitle

\section{Introduction}\label{section1}

One of the most significant progresses in the theory of sampling and reconstruction was achieved by Whittaker, Kotelnikov, and Shannon, who showed that any band-limited signal, i.e., a signal whose Fourier transform is compactly supported, can be fully recovered from its regularly spaced samples (see \cite{srvbut}). This result is widely known as the \textbf{WKS sampling theorem}. In 1968, Butzer and Stens \cite{but1} extended this theorem to more general, not necessarily band-limited signals. Since then, numerous mathematicians have made substantial contributions in this direction (see \cite{butzer2,tam}).


A major milestone in approximation theory is the famous Weierstrass Approximation Theorem. In 1912, S. Bernstein provided a constructive proof that sparked significant interest in the study of positive linear operators. However, a limitation of such operators is their reliance on uniformly spaced data points, whereas real-world data is often non-uniform. To address this issue, Ostrowski et al. \cite{ostrowsky} introduced a series representation for the class of Mellin band-limited functions using exponentially spaced data. This reconstruction formula, known as the \textit{exponential sampling formula}, is defined as follows. For $f:\mathbb{R}^{+} \rightarrow \mathbb{C}$ and $c \in \mathbb{R},$ the exponential sampling formula is given by (see \cite{butzer3})
\begin{equation} \label{expformula}
(E_{c,T}f)(s):= \sum_{k=-\infty}^{\infty} \lin_{\frac{c}{T}}(e^{-k}x^{T}) f(e^{\frac{k}{T}}),
\end{equation}
where $\lin_{c}(s)= \dfrac{u^{-c}}{2\pi i} \dfrac{u^{\pi i}-u^{-\pi i}}{\log c} = u^{-c}~ \sinc(\log u)$ with continuous extension $\lin_{u}(1)=1.$ Moreover, if $f$ is Mellin band-limited to $[-T,T],$ then $(E_{c,T}f)(s)=f(s)$ for all $u \in \mathbb{R}_{+}.$

Exponentially spaced data have numerous applications in optical physics, radio astronomy, and engineering, including Fraunhofer diffraction, polydispersity analysis via photon correlation spectroscopy, and neuronal scattering (see \cite{gori,ostrowsky,bertero}). This underscores the motivation for modifying and generalizing the operator \eqref{expformula}. Butzer and Jansche \cite{butzer5} employed Mellin analysis to explore the exponential sampling formula, demonstrating that the Mellin transform is a suitable tool for sampling and approximation problems involving exponentially spaced data. The foundational work on Mellin transformation by Mamedov \cite{mamedov}, along with subsequent contributions by Butzer et al. \cite{butzer3,butzer5}, laid the foundation for these studies. Further developments in Mellin theory can be found in \cite{bardaro1,bardaro3}.

To approximate functions that are not necessarily Mellin band-limited, Bardaro et al. \cite{bardaro7} introduced a generalized exponential sampling series using an appropriate kernel. For $z \in \mathbb{R}_{+}$ and $n>0,$ the generalized exponential sampling series is defined as
\begin{equation} \label{genexp}
(\mathcal{S}_{\mathrm{n}}^{\varphi}f)(z)= \sum_{k=- \infty}^{\infty} \varphi(z^{n}e^{-k}) f( e^{\frac{k}{\mathrm{n}}}),
\end{equation}
for any $ f: \mathbb{R}_{+} \rightarrow \mathbb{R}$ such that the series \eqref{genexp} converges absolutely. Various approximation properties of the operator family \eqref{genexp} have been studied in \cite{bardaro11}, and the behavior of such sampling operators in weighted spaces has been extensively explored in \cite{Acar2023a,pradhan2025,Acar2025}. 

It is evident that summation-type operators are not suitable for approximating Lebesgue integrable functions. In this context, Kantorovich \cite{Kant} and Durrmeyer \cite{Durr} operators, which are modifications of classical Bernstein-type operators, were introduced. Similarly, Bardaro and Mantellini \cite{bardaro2021durrmeyer} introduced Durrmeyer-type exponential sampling operators, while Aral et al. \cite{aral} proposed Kantorovich-type exponential sampling operators.

For an integrable function $\mathfrak{h},$ the Kantorovich exponential sampling operator is defined as
$$K_{n}^{\phi}(\mathfrak{h})(z):=n\sum\limits_{k=-\infty}^{\infty}\phi\left(z^{n}e^{-k}\right)\int\limits_{\frac{k}{n}}^{\frac{k+1}{n}}\mathfrak{h}\big(e^u\big)du,$$ and the Durrmeyer-type exponential sampling operator is defined by $$D_{n}^{\phi,\Psi}(\mathfrak{h})(z):=n\sum\limits_{k=-\infty}^{\infty}\phi\left(z^{n}e^{-k}\right)\int\limits_{0}^{\infty}\Psi\left(z^{n}e^{-k}\right)\mathfrak{h}(u)\frac{du}{u}.$$ 

Over time, numerous modifications and generalizations of such operators have been developed (see \cite{cai2024convergence,costarelli2023quantitative,costarelli2023quantitative,ang2,bajpeyi2022approximation,prashant,tam, coroianu2024approximation, costarelli2023convergence}).

Recently, Coroianu and Gal \cite{CG2010} proposed an innovative approximation method known as the max-product approach, which improves the level of approximation achievable by a class of linear operators. In this approach, the sums or series in discrete linear operators are replaced by the maximum or supremum over the same index sets, thereby converting linear operators into nonlinear ones and often attaining a higher order of approximation compared to the linear case (see \cite{CG2011,CG2012}). Notable work in this direction is the work of Angamuthu~\cite{Ang1}, on multivariate max-product operators, and Costarelli and Sambucini\cite{costarelli2018approximation}, on Kantorovich max-product operators in Orlicz spaces, and Bajpeyi et al.\cite{bajpeyi2024exponential}, on exponential sampling type Kantorovich max-product operators. Motivated by the improved approximation properties of max-product operators and the stability enhancements provided by Durrmeyer modifications, it is natural to investigate Durrmeyer variants of max-product operators as a framework for advancing approximation results.

Max–min operators have also attracted significant interest, particularly in the theory of fuzzy approximation. Gokcer and Duman \cite{gokcer2020} developed a general framework for max–min operators, effectively approximating non-negative continuous functions and providing error estimates for Hölder continuous functions. This was extended by Gökcers and Aslan \cite{gokcer2022a} to Kantorovich-type max–min operators, which converge quantitatively and find applications in signal processing and integral equations, thereby allowing approximation of a wider class of Lebesgue integrable functions. Recent devlopments in this direction can be found in \cite{vayeda2025approximation, pradhan2025a}.


Building on this foundation, the present study investigates the approximation behavior of Durrmeyer-type Max-product and Max–min exponential sampling operators. We establish quantitative error bounds for both operators in approximating Lebesgue integrable functions. Furthermore, to illustrate their convergence behavior, we provide examples using various kernel functions along with graphical representations of the results.  

The work is structured as follows. Section \ref{section2} describes the fundamentals of the Durrmeyer-type max-product and max-min generalized exponential sampling with some auxiliary results. The third section \ref{sec3} is dedicated to the Durrmeyer variant of Max product exponential sampling operator. The well-definedness of the operators as well as their pointwise and uniform convergence, have been studied. Additionally, the error in the approximation has been quantitatively estimated using the logarithmic modulus of continuity. Section \ref{sec4} addresses the corresponding results for the Durrmeyer variant of the Max–min exponential sampling operator. Finally, Section \ref{sec5} provides graphical illustrations and error calculations for selected test functions to demonstrate the convergence behavior of both newly defined operators using two different kernels.

\section{Preliminaries}\label{section2} 

Henceforth, we use the notation: $\mathbb{R_{+}}:=(0,\infty),\, \mathbb{R}_{+}^{0}:=[0,\infty)$. Let $I=[a,b]$ denote a compact interval in $\mathbb{R}_{+}$.

Define the \emph{metric} on $\mathbb{R}_+$ as: $d(s,t) = |s - t|$.

\begin{itemize}
    \item A function $\mathfrak{h}:\mathbb{R}_{+} \rightarrow \mathbb{R}$ is said to be \emph{log-continuous} at a point $w\in\mathbb{R}_{+}$ if, for every $\epsilon > 0$, there exists $\varsigma > 0 $ such that \[d(\log(s), \log(z)) \leq \varsigma \quad \Longrightarrow  \quad d(\mathfrak{h}(u), \mathfrak{h}(z)) < \epsilon. \] 
    \item The function $\mathfrak{h}:\mathbb{R}_{+} \rightarrow \mathbb{R}$ is \emph{log-uniformly continuous} on $\mathbb{R}_{+} $ if, for every $\epsilon > 0$, there exists $\varsigma > 0 $ such that \[d(\log(w_1), \log(w_2)) \leq \varsigma \quad \Longrightarrow  \quad d(\mathfrak{h}(w_1), \mathfrak{h}(w_2)) < \epsilon , \quad \forall w_1, w_2 \in\mathbb{R}_+.\]

\item $\mathscr{L}^{\infty}(\mathbb{R}_{+}) := \{\mathfrak{h}: \mathbb{R}_{+} \to \mathbb{R}\,\, \vert\,     \operatorname{ess\,sup}\limits_{z\in\mathbb{R}_{+}} |\mathfrak{h}|< \infty \}$
\item For $1 \leq p < \infty$, $\mathscr{L}^{p}(\mathbb{R}_{+})$ denotes the space of all $p$-integrable functions on $\mathbb{R}_{+}$ with respect to the Lebesgue measure.
    \item  $\mathscr{LU}_{b}(\mathbb{R}_{+})\, := \, \{ \mathfrak{h}: \mathbb{R}_{+} \to \mathbb{R}\, \vert  \text{ $\mathfrak{h}$ is log-uniformly continuous and bounded}\}$, equipped with the norm $\|\mathfrak{h}\|_{\infty} \,:= \,\sup\limits_{w\in \mathbb{R}_{+}} |\mathfrak{h}(z)|.$     
\end{itemize}
It is clear that in any compact subset of $\mathbb{R}_{+}$, log-uniform continuity  and uniform continuity are equivalent.

For any index set $ \Lambda \subseteq \mathrm{Z},$ we define 
\begin{equation}
\bigvee_{j \in \Lambda} T_j :=
\begin{cases}
\max\{ T_j : j \in \Lambda \}, & \text{if $\Lambda$
 is finite}, \\[0.5em]
 \sup\{ T_j : j \in \Lambda \}, & \text{if $\Lambda$ is infinite}.
\end{cases}
\end{equation}
We now introduce some kernel functions of interest.

Let $\Phi :\mathbb{R}_{+} \rightarrow \mathbb{R} $ be a bounded measurable function satisfying the following properties:
\begin{description}
    \item[$(\Phi.{1}):$] The discrete absolute moment of order \( r \) ($\mathfrak{m}_{r}(\Phi)$) is finite for $r=2,$
where    \[\mathfrak{m}_{r}(\Phi) = \sup_{z \in \mathbb{R}_{+}} \bigvee_{k \in \mathbb{Z}}
        |\Phi(ze^{-k})|\,|\log z-k|^{r}.\]

    \item[$(\Phi.{2}):$]
    \[\inf_{z \in [1,e]} \Phi(z) =: \vartheta_{z}.\]
\end{description}
Let $\widetilde{\Phi}$ denote the class of continuous functions $\Phi$ that satisfy conditions \((\Phi.1)\) and \((\Phi.2)\).
\vspace{0.4em}
The following lemmas are presented here without proofs.
\begin{lemma}\label{lma1}(see\cite{Ang1})
    Let $\Phi$ be a bounded function satisfying the condition $(\Phi.{1})$, and let $\mu > 0$. Then
    \[\mathrm{m}_{\nu}(\Phi) < \infty, \quad\text{ for all} \quad 0 \leq \nu \leq \mu.\]
\end{lemma}
\begin{lemma}\label{lma2}(see\cite{Ang1})    
  Let $\Phi\in\widetilde{\Phi}$ satisfy $ \mathfrak{m}_{\nu}(\Phi) < \infty$ for all $\nu > 0$. Then, for any $\varsigma >0$,   
  \begin{equation*}
  \bigvee_{\substack{k\in \mathbb{Z}\\|k-\log(z)| >  \mathrm{n} \varsigma}}|\Phi(ze^{-k})| = \mathcal{O}(\mathrm{n}^{-\nu}), \quad \text{as $ \mathrm{n} \rightarrow \infty $}
  \end{equation*}  
   uniformly with respect to $ z \in \mathbb{R_{+}}$.
\end{lemma}
\begin{lemma}\label{lm3}(see\cite{Ang1})
    Let $\Phi : \mathbb{R_{+}}\rightarrow \mathbb{R}$ be a kernel satisfying condition $(\Phi.{2})$. Then the following hold:
    \begin{enumerate}
        \item For any $z \in \mathbb{R}_{+}$,
        \begin{equation*}
            \bigvee_{k\in \mathbb{J}_{\mathrm{n}} } |\Phi(z^{\mathrm{n}}e^{-k})| \geq \vartheta_{z},
        \end{equation*} 
        where $\mathrm{J}_{\mathrm{n}} = \{ k \in \mathbb{Z} : k = \lceil{\mathrm{n}\log a}\rceil.....,\lfloor{\mathrm{n}\log b}\rfloor \}.$
        \item For any $z \in \mathbb{R_{+}}$,  \begin{equation*}\bigvee_{k\in \mathrm{Z}} |\Phi(z^{n}e^{-k})| \geq \vartheta_{z}. \end{equation*}
    \end{enumerate}  
\end{lemma}

Let $\Psi : \mathbb{R}_{+} \to \mathbb{R}$ be a function satisfying the following conditions:
\begin{description}
    \item[$(\Psi.1):$] \quad $\displaystyle \int_{0}^{\infty} \Psi(z) \, \frac{dz}{z} = 1$.
    \item[$(\Psi.2):$] \quad $\displaystyle \mathrm{M}_{0}(\Psi) := \int_{0}^{\infty} |\Psi(z)| \, \frac{dz}{z} < \infty$.
\end{description}

\begin{definition}
    For any $r \in \mathbb{Z}_{+}\cup{\{0\}}$, the continuous algebraic moment of $\Psi$ of order $r$ is defined as
    \[\widehat{\mathrm{M}}_{r}(\Psi):= \int\limits_{0}^{\infty} \Psi(z) \, (\log{z})^{r}\, \frac{dz}{z}.\]
\end{definition}

\begin{definition}
    For any $r \in \mathbb{Z}_{+}\cup{\{0\}}$, the continuous absolute moment of $\Psi$ of order $r$ is defined by:
    \[\mathrm{M}_{r}(\Psi):= \int\limits_{0}^{\infty} |\Psi(z)| \, |\log z|^{r}\, \frac{dz}{z}.\] 
\end{definition}
\begin{remark}
    Let $r_{1}, r_{2} \in  \mathbb{Z}_{+}$ such that $r_{1}< r_{2}.$ Then $\mathrm{M}_{r_1}(\Psi) < \infty$ whenever $\mathrm{M}_{r_2}(\Psi) < \infty.$
\end{remark}
\begin{lemma}\label{le5}
Suppose $\Psi$ is a real valued function on $(0,\infty)$ satisfying $(\Psi.2)$. Fix some $\delta>0$ and $k\in\mathbb{Z}$. Then, for every $\varepsilon>0$, there exists $n_{0}\in\mathbb{N}$ such that for all $n\ge n_{0}$,
\[
n\,\int\limits_{\;|u^{n}e^{-k}|>e^{\tfrac{n\delta}{2}}}
\left|\Psi(u^{n}e^{-k})\,\right|\frac{du}{u}<\varepsilon.
\]
\end{lemma}
Let $\widetilde{\psi}$ denote the set of all continuous functions $\Psi$ that satisfy both \((\Psi.1)\) and \((\Psi.2)\).

\begin{lemma}\label{kerlma}
For $\Phi\in\widetilde{\phi}$, $\Psi\in\widetilde{\psi}$ and $z\in\mathbb{R}_{+},$ we have
    \[ \bigvee\limits_{k\in \mathbb{Z}} \Phi(z^{n}e^{-k})\quad n \int\limits_{0}^{\infty} \Psi(u^{n}e^{-k})\, \frac{du}{u}\, \geq \,\vartheta_{z}.\]
\end{lemma}
\begin{proof} We begin with
    \begin{align*}
        \bigvee\limits_{k\in \mathbb{Z}}  \Phi(z^{n}e^{-k})\quad n \int\limits_{0}^{\infty} \Psi(u^{n}e^{-k})\, \frac{du}{u} = & \bigvee\limits_{k\in \mathbb{Z}} \left[ n \int\limits_{0}^{\infty} \Psi(u^{n}e^{-k})\, \frac{du}{u}\right]\Phi(z^{n}e^{-k}).\\
    \end{align*}

Setting $ u^{n}e^{-k} = t$ and using ($\Psi .1$) together with Lemma \ref{lm3}, we obtain 
     \begin{align*}
        \bigvee\limits_{k\in \mathbb{Z}}  \Phi(z^{n}e^{-k})\quad n \int\limits_{0}^{\infty} \Psi(u^{n}e^{-k})\, \frac{du}{u}= & \quad \bigvee\limits_{k\in \mathbb{Z}} \left[ \int\limits_{0}^{\infty} \Psi(t)\, \frac{dt}{t}\right]\Phi(z^{n}e^{-k})\\
    = & \quad \bigvee\limits_{k\in \mathbb{Z}}\, \Phi(z^{n}e^{-k})\\
        \geq & \quad \vartheta_{z}.
    \end{align*}
    This completes the proof.
\end{proof}

\vspace{0.4em}
Next, we establish the well-definedness of the Max-Min type operators with the help of the following results.

\begin{lemma}\label{lm5} (see \cite{aslan2025})                                                        If $\bigvee\limits_{m \in \mathbb{Z}}s_{m}$ and $\bigvee\limits_{m \in \mathbb{Z}}t_{m} $ are finite, then the following inequality holds:
\[\bigvee\limits_{m \in \mathbb{Z}}s_{m}  - \bigvee\limits_{m \in \mathbb{Z}}t_{m} \leq \bigvee\limits_{m \in \mathbb{Z}}\left|s_{m}-t_{m} \right| .\]                                    
\end{lemma}

\begin{lemma}\label{lm6} (see \cite{bede2016}) 
    For $\mathfrak{r}, \mathfrak{p}, \mathfrak{q}\in [0,1]$, we have:
    \[\left|(\mathfrak{r} \land \mathfrak{p}) - (\mathfrak{r} \land \mathfrak{q})\right|\leq \mathfrak{r} \land \left| \mathfrak{p}-\mathfrak{q}\right|,\quad \text{where }  \mathfrak{r} \land \mathfrak{p}= \min\{\mathfrak{r},\mathfrak{p}\}. \]
\end{lemma}

\begin{lemma}\label{lm7}(see \cite{bede2008})
    For any $\mathfrak{r},\mathfrak{p},\mathfrak{q} \geq 0$, the following inequality holds: \[ (\mathfrak{r}\land \mathfrak{q}) + (\mathfrak{p} \land \mathfrak{q}) \geq  (\mathfrak{r}+\mathfrak{p})\land \mathfrak{q} .\]
\end{lemma}

\begin{lemma}\label{lm8}(see\cite{vayeda2025approximation})
   For a finite set  $ I\subset\mathbb{Z}_+ $ and a positive real number $ \lambda $, we have \[
\lambda \bigvee_{m \in I} (\mathfrak{s}_m \wedge \mathfrak{t}_m) = \bigvee_{m \in I} (\lambda \mathfrak{s}_m \wedge \lambda \mathfrak{t}_m),
\] where $ \{\mathfrak{s}_m\}_{m \in I} $ and $ \{\mathfrak{t}_m\}_{m \in I} $ are sequences in $ [0,1] $.
\end{lemma}

\begin{definition}\label{def3}
    For a given $w\in \mathbb{R}_{+}\subseteq \mathbb{R}_{+}$ and $ n \in \mathbb{N}$ and $\tau > 0$, we define $\mathscr{B}_{\tau, n} $ as: 
    \[\mathscr{B}_{\tau, n} : = \{ k = \lceil{n\log{a}}\rceil,.......\lfloor{n\log{b}}\rfloor  : \, \left| \frac{k}{n}- \log{z}\right| \leq \tau \}.\]
\end{definition}

\section{The Max-product Durrmeyer type Exponential sampling operators} \label{sec3}
This section is devoted to the study of pointwise and uniform convergence, together with the rate of approximation, for the Max-product Durrmeyer-type exponential sampling operators.
\begin{definition}\label{Maxdef1}
    Let $\mathfrak{h}: \mathbb{R}_{+} \rightarrow \mathbb{R}$ be a bounded and $\mathscr{L}^1$-integrable function on $\mathbb{R}_{+}$. Then, the Max-product Durrmeyer-type exponential sampling operators associated with $\mathfrak{h}$ with respect to $\Phi$ and $\Psi$, are defined by
    \[ \mathscr{D}_{n,\Phi,\Psi}^{M}(\mathfrak{h})(z) := \frac{\bigvee\limits_{k\in \mathbb{Z}} \Phi(z^{n}e^{-k})\quad n \int\limits_{0}^{\infty} \Psi(u^{n}e^{-k})\,\, \mathfrak{h}(u)\,\frac{du}{u}} {\bigvee\limits_{k\in \mathbb{Z}} \Phi(z^{n}e^{-k})\quad n \int\limits_{0}^{\infty} \Psi(u^{n}e^{-k})\, \frac{du}{u}}.\]    
\end{definition}
Since the function $\mathfrak{h}$ is bounded and integrable and the  kernels $\Phi \text{ and } \Psi$ are integrable, it follows that the defined operator  $\mathscr{D}_{n,\Phi,\Psi}^{M}(\mathfrak{h})(z)$is bounded and therefore well-defined.

\begin{lemma}\label{Maxlma1}
    Let $\mathfrak{h},\mathrm{g} $ be two locally integrable functions on $\mathbb{R}_{+}$. Then, the Max-product exponential sampling operators $\mathscr{D}^{M}_{n,\Phi,\Psi}$ satisfy the following properties: 
    \begin{enumerate}
        \item If  $ \mathfrak{h}(z) \leq \mathfrak{g}(z),$ then $$\mathscr{D}_{n,\Phi, \Psi}^{M}(\mathfrak{h}) \leq \mathscr{D}_{n,\Phi,\Psi}^{M}(\mathfrak{g}).$$
        \item For all $z \in \mathbb{R}_{+}$ $$\mathscr{D}_{n,\Phi,\Psi}^{M}(\mathfrak{h+g})(z) \leq \mathscr{D}_{n,\Phi,\Psi}^{M}(\mathfrak{h})(z)+ \mathscr{D}_{n,\Phi,\Psi}^{M}(\mathfrak{g})(z).$$
        \item For all $z \in I$, $$|\mathscr{D}_{n,\Phi,\Psi}^{M}(\mathfrak{h})(z)- \mathscr{D}_{n,\Phi,\Psi}^{M}(\mathfrak{g})(z)| \leq \mathscr{D}_{n,\Phi,\Psi}^{M}(\mathfrak{\\|h-g\\|})(z). $$
        \item For every $\lambda > 0$ and all $z \in \mathbb{R}_{+}$, $$\mathscr{D}_{n,\Phi,\Psi}^{M}(\mathrm{\lambda \,\mathfrak{h}})(z) = \lambda \hspace{0.1in} \mathscr{D}_{n,\Phi,\Psi}^{M}(\mathfrak{h})(z).$$
    \end{enumerate}
\end{lemma}
\begin{proof}
     We can easily prove (1), (2) and (4) by using definitions of operators $\mathscr{D}_{n,\Phi,\Psi}^{M}(\mathfrak{h})$. To prove the third part, we use the inequalities $ \mathfrak{h}(z) \leq |\mathfrak{h}(z) - \mathfrak{g}(z)| + \mathfrak{g}(z) $ and $ \mathfrak{g}(z) \leq |\mathfrak{g}(z) - \mathfrak{h}(z)| + \mathfrak{h}(z)\,\, $ together with  properties (1) and (2), to obtain  \begin{equation*}|\mathscr{D}_{n,\Phi,\Psi}^{M}(\mathfrak{h})(z)-\mathscr{D}_{n,\Phi,\Psi}^{M}(\mathfrak{g})(z)| \leq \mathscr{D}_{n,\Phi,\Psi}^{M}({\\|\mathfrak{h}-\mathfrak{g}\\|})(z), \quad \text{for all}~~ u\in I.\end{equation*}
\end{proof}


\begin{theorem}\label{maxthm1}
Let $\mathfrak{h}:I \to \mathbb{R}\,$ be bounded and $\mathscr{L}^1$ be integrable. Then, at any point of log-continuity $w\in I$ \[ \lim\limits_{n\to \infty} \mathscr{D}_{n,\Phi, \Psi}^{M}\,(\mathfrak{h}) (z) = \mathfrak{h}(z) \]
Moreover, if $\mathfrak{h} \in \mathscr{LU}_{b}(I)$, then the convergence is uniform, i.e.,
\[\lim\limits_{n\to \infty} \left\|\mathscr{D}_{n,\Phi, \Psi}^{M}\,(\mathfrak{h})  - \mathfrak{h}\right\|_{\infty} = 0. \]
 where $\| .\|$ denotes the supremum norm.
\end{theorem}

\begin{proof}
    From Definition \ref{Maxdef1}, we can write 
    \begin{align*}
      \left|\right.&\left.\mathscr{D}_{n,\Phi,\Psi}^{M}(\mathfrak{h})(z) - \mathfrak{h}(z)\right|  = \left| \frac{\bigvee\limits_{k\in \mathbb{Z}} \Phi(z^{n}e^{-k}) n \int\limits_{0}^{\infty} \Psi(u^{n}e^{-k})\mathfrak{h}(u)\frac{du}{u}} {\bigvee\limits_{k\in \mathbb{Z}} \Phi(z^{n}e^{-k}) n \int\limits_{0}^{\infty} \Psi(u^{n}e^{-k}) \frac{du}{u}}  -  \mathfrak{h}(z) \right|\\
        = & \left| \frac{\bigvee\limits_{k\in \mathbb{Z}} \Phi(z^{n}e^{-k}) n \int\limits_{0}^{\infty} \Psi(u^{n}e^{-k}) \mathfrak{h}(u)\frac{du}{u}} {\bigvee\limits_{k\in \mathbb{Z}} \Phi(z^{n}e^{-k}) n \int\limits_{0}^{\infty} \Psi(u^{n}e^{-k}) \frac{du}{u}}  - \frac{\bigvee\limits_{k\in \mathbb{Z}} \Phi(z^{n}e^{-k}) n \int\limits_{0}^{\infty} \Psi(u^{n}e^{-k}) \mathfrak{h}(z)\,\frac{du}{u}} {\bigvee\limits_{k\in \mathbb{Z}} \Phi(z^{n}e^{-k}) n \int\limits_{0}^{\infty} \Psi(u^{n}e^{-k}) \frac{du}{u}} \right|\\
        \le &~ \frac{\bigvee\limits_{k\in \mathbb{Z}} \big|\Phi(z^{n}e^{-k})\big| n \int\limits_{0}^{\infty} \big|\Psi(u^{n}e^{-k})\big||\mathfrak{h}(u)-\mathfrak{h}(z)|\frac{du}{u}} {\left|\bigvee\limits_{k\in \mathbb{Z}} \Phi(z^{n}e^{-k}) n \int\limits_{0}^{\infty} \Psi(u^{n}e^{-k}) \frac{du}{u}\right|}.
    \end{align*}
Using Lemma~\ref{kerlma} and property (3) of  Lemma \ref{Maxlma1}, we obtain
\begin{align*}
   \left|\mathscr{D}_{n,\Phi,\Psi}^{M}(\mathfrak{h})(z) - \mathfrak{h}(z)\right| \leq & \, \frac{1}{\vartheta_{z}}  \bigvee\limits_{k\in \mathbb{Z}} \Phi(z^{n}e^{-k})\quad n \int\limits_{0}^{\infty} \Psi(u^{n}e^{-k})\,\,\left|  (\mathfrak{h}(u)-\mathfrak{h}(z))\right| \, \frac{du}{u}.
\end{align*}
Let $\varepsilon>0$. By the log-continuity of $\mathfrak{h}$ at $z$, there exists $\delta(\epsilon)>0$ such that $$|\log(u)-\log(z)|=\log\left(\frac{u}{z}\right)<\delta\implies |h(u)-h(z)|<\epsilon.$$ Hence,
\begin{align*}
    \left|\right.&\left.\mathscr{D}_{n,\Phi,\Psi}^{M}(\mathfrak{h})(z) - \mathfrak{h}(z)\right| \leq  \frac{1}{\vartheta_{z}}  \bigvee\limits_{k\in \mathbb{Z}} \big|\Phi(z^{n}e^{-k})\big| n \int\limits_{0}^{\infty} \big|\Psi(u^{n}e^{-k})\big|\left|  (\mathfrak{h}(u)-\mathfrak{h}(z))\right|  \frac{du}{u}.\\
    \leq & \max \left\{\frac{1}{\vartheta_{z}}\bigvee_{\substack{k\in \mathbb{Z}\\ |z^{n}e^{-k}|\leq e^\frac{n\delta}{2} }} \big|\Phi(z^{n}e^{-k})\big| n \int\limits_{0}^{\infty} \big|\Psi(u^{n}e^{-k})\big|\left|  (\mathfrak{h}(u)-\mathfrak{h}(z))\right| \frac{du}{u}, \right. \\ &\left.\frac{1}{\vartheta_{z}}\bigvee_{|z^{n}e^{-k}|> e^ \frac{n\delta}{2} }  \big|\Phi(z^{n}e^{-k})\big| n \int\limits_{0}^{\infty} \big|\Psi(u^{n}e^{-k})\big|\left|  (\mathfrak{h}(u)-\mathfrak{h}(z))\right|\frac{du}{u}\right\}\\
    = & \max\{D_1,D_2\}. \quad \text{(say)}
\end{align*}


Estimation of $D_1$. For sufficiently large $n$, using the log-continuity of $\mathfrak{h}$ and Lemma \ref{le5} (with $u^{n} e^{-k}=t $), we get
\begin{align*}
    D_1=&\frac{1}{\vartheta_{z}}\bigvee_{\substack{k\in \mathbb{Z}\\ |z^{n}e^{-k}|\leq e^\frac{n\delta}{2} }} \big|\Phi(z^{n}e^{-k})\big| n \int\limits_{0}^{\infty} \big|\Psi(u^{n}e^{-k})\big|\left|  (\mathfrak{h}(u)-\mathfrak{h}(z))\right|\frac{du}{u}\\
    \leq & \frac{1}{\vartheta_{z}}\bigvee_{\substack{k\in \mathbb{Z}\\ |z^{n}e^{-k}|\leq e^\frac{n\delta}{2} }} \big|\Phi(z^{n}e^{-k})\big| n \int\limits_{|u^{n}e^{-k}|\leq e^\frac{n\delta}{2}} \big|\Psi(u^{n}e^{-k})\big|\left|  (\mathfrak{h}(u)-\mathfrak{h}(z))\right| \frac{du}{u}\\ & + \frac{1}{\vartheta_{z}}\bigvee_{\substack{k\in \mathbb{Z}\\ |z^{n}e^{-k}|\leq e^\frac{n\delta}{2} }} \big|\Phi(z^{n}e^{-k})\big| n \int\limits_{|u^{n}e^{-k}|> e^\frac{n\delta}{2}} \big|\Psi(u^{n}e^{-k})\big|\left|  (\mathfrak{h}(u)-\mathfrak{h}(z))\right|  \frac{du}{u}\\
    \leq & \frac{\epsilon}{\vartheta_{z}} \bigvee_{k\in \mathbb{Z}} |\Phi(z^{n}e^{-k})| \int\limits_{|t | <e^\frac{n\delta}{2}} |\Psi(t)| \frac{dt}{t}+\frac{2 \|h\|_{\infty}}{\vartheta_{z}}\bigvee_{k\in \mathbb{Z}} |\Phi(z^{n}e^{-k})|\int\limits_{|t | > e^\frac{n\delta}{2}} |\Psi(t)| \frac{dt}{t}\\
    \leq & \frac{\epsilon}{\vartheta_{z}}\mathfrak{m}_{0}(\Phi)\mathrm{M}_{0}(\Psi) + \frac{2\epsilon\|\mathfrak{h}\|_{\infty}}{\vartheta_{z}}\mathfrak{m}_{0}(\Phi).
\end{align*}

Estimation of $D_2$. Using $(\Psi.2)$ and Lemma \ref{lma2}, we obtain
\begin{align*}
         D_2=& \frac{1}{\vartheta_{z}}\bigvee_{|z^{n}e^{-k}|> e^ \frac{n\delta}{2} }  \big|\Phi(z^{n}e^{-k})\big|\quad n \int\limits_{0}^{\infty} \big|\Psi(u^{n}e^{-k})\big|\left|  (\mathfrak{h}(u)-\mathfrak{h}(z))\right| \frac{du}{u} \\
         \leq & \frac{2 \|h\|_{\infty}}{\vartheta_{z}} \bigvee_{|z^{n}e^{-k}|> e^ \frac{n\delta}{2} } \big|\Phi(z^{n}e^{-k})\big|\quad n \int\limits_{0}^{\infty} \big|\Psi(u^{n}e^{-k})\big|  \frac{du}{u}\\
         \leq & \frac{2\epsilon \|h\|_{\infty}}{\vartheta_{z}} \mathrm{M}_{0}(\Psi).
\end{align*}


Combining the above estimates for $D_1$ and $D_2$, we obtain the desired result:
 \[ \lim\limits_{n\to \infty} \mathscr{D}_{n,\Phi, \Psi}^{M}\,(\mathfrak{h}) (z) = \mathfrak{h}(z) \]
 Moreover, if $\mathfrak{h} \in \mathscr{LU}_{b}(I)$, then the convergence is uniform, completing the proof.
\end{proof}

\begin{definition}\label{LMSDef}(see \cite{Ang1})
   For $\mathrm{h} \in \mathscr{LU}_{b}(\mathbb{R}_{+})$, we define the logarithmic modulus of continuity by
   \[{\mho}_{\mathscr{I}}(\mathrm{h},\varsigma):=\sup\limits_{s,t \in \mathbb{R}_{+}} \{\left|\mathrm{h}(s) - \mathrm{h}(t)\right| : | \log{s}- \log{t}| \leq \varsigma, \varsigma >0 \}.\]
  Some fundamental properties of ${\mho}_{\mathbb{R}_{+}}(\mathrm{h},\varsigma)$ are as follows:
   \begin{enumerate}
       \item[(a)] $\lim\limits_{\varsigma \to 0}{\mho}_{\mathbb{R}_{+}}(\mathrm{h},\varsigma) \to 0 $.
       \item[(b)] $ {\mho}_{\mathbb{R}_{+}}(\mathrm{h},\beta \varsigma)\leq (\beta +1) {\mho}_{\mathbb{R}_{+}}(\mathrm{h},\varsigma)$, for all $\beta > 0$.
       \item[(c)] $ \left|\mathrm{h}(s) - \mathrm{h}(t)\right| \leq {\mho}_{\mathbb{R}_{+}}(\mathrm{h},\varsigma) \left( 1 + \frac{\| \log{s} - \log{t}\|}{\varsigma}\right),$ for all $ s,t \in \mathbb{R}_{+}.$
   \end{enumerate}
\end{definition}

\begin{theorem}
     Let $\Phi\in\widetilde{\phi}$ and $\Psi\in\widetilde{\psi}$ satisfy $\mathfrak{m}_{1}(\Phi)<\infty$ and $\mathrm{M}_{1}(\Psi)<\infty$. If $h\in\mathscr{LU}_{b}(I)$, then for every $\varsigma>0$, the following inequality holds:      
    \begin{align*} 
    \left|\right.&\left. \mathscr{D}_{n,\Phi,\psi}^{M}(\mathrm{h})(z)- \mathrm{h}(z)\right|\\ &\qquad \leq  \frac{\mho_{\mathscr{I}}(h,\varsigma)}{\vartheta_{z}}  \mathfrak{m}_{0}(\Phi)\cdot\mathrm{M}_{0}(\Psi) + \frac{\mho_{\mathscr{I}}(h,\varsigma)}{n \, \vartheta_{z}\cdot \varsigma} \left\{ \mathfrak{m}_{0}(\Phi)\cdot \mathrm{M}_{1}(\Psi) + \mathfrak{m}_{1}(\Phi)\cdot \mathrm{M}_{0}(\Psi) \right\}.  \end{align*} 
\end{theorem}

\begin{proof}
     Using Definition \ref{Maxdef1}, Lemma \ref{kerlma} and properties (3) of Lemma  \ref{Maxlma1}  we have
     \begin{align*}
        \left|\mathscr{D}_{n,\Phi,\psi}^{M}(\mathrm{h})(z) - \mathrm{h}(z)\right| \leq & \frac{1}{\vartheta_{z}}  \bigvee\limits_{k\in \mathbb{Z}} |\Phi(z^{n}e^{-k})| \quad n \int\limits_{0}^{\infty} |\psi(u^{n}e^{-k})|\left|  (\mathrm{h}(u)-\mathrm{h}(z))\right| \frac{du}{u}\\
        \leq & \frac{1}{\vartheta_{z}}  \bigvee\limits_{k\in \mathbb{Z}} |\Phi(z^{n}e^{-k})| n \int\limits_{0}^{\infty} |\psi(u^{n}e^{-k})|\mho_{\mathscr{I}}\left(h, \left|\log\left(\frac{u}{z}\right)\right|\right) \frac{du}{u}.
     \end{align*} 
     Now, for any $\varsigma > 0$, applying the property (b) of Definition \ref{LMSDef}, we obtain  
     \begin{align*}
        \left| \right.&\left.\mathscr{D}_{n,\Phi,\psi}^{M}(\mathrm{h})(z)- \mathrm{h}(z)\right| \\&\quad\leq  \frac{1}{\vartheta_{z}}  \bigvee\limits_{k\in \mathbb{Z}} |\Phi(z^{n}e^{-k})|  n \int\limits_{0}^{\infty} |\psi(u^{n}e^{-k})| \left( 1 + \frac{|\log\left(\frac{u}{z}\right)|}{\varsigma}\right)  \mho_{\mathscr{I}}(h,\varsigma)  \frac{du}{u}\\
         &\quad\leq E_{1} + E_{2}.
     \end{align*} 
     For $E_{1}$, we have 
     \begin{align*}
         E_{1} &=  \frac{1}{\vartheta_{z}}  \bigvee\limits_{k\in \mathbb{Z}} |\Phi(z^{n}e^{-k})|  n \int\limits_{0}^{\infty} |\psi(u^{n}e^{-k})| \mho_{\mathscr{I}}(h,\varsigma)  \frac{du}{u}\\
          &= \frac{\mho_{\mathscr{I}}(h,\varsigma)}{\vartheta_{z}}  \bigvee\limits_{k\in \mathbb{Z}} |\Phi(z^{n}e^{-k})| \quad n \int\limits_{0}^{\infty} |\psi(u^{n}e^{-k})| \frac{du}{u}\\
        & \leq  \frac{\mho_{\mathscr{I}}(h,\varsigma)}{\vartheta_{z}}  \mathfrak{m}_{0}(\Phi)\cdot\mathrm{M}_{0}(\Psi).
     \end{align*}
     Next, for $E_{2}$ we get
     \begin{align*}
         E_{2} &= \frac{1}{\vartheta_{z}}  \bigvee\limits_{k\in \mathbb{Z}} |\Phi(z^{n}e^{-k})| n \int\limits_{0}^{\infty} |\psi(u^{n}e^{-k})| \left(\frac{|\log\left(\frac{u}{z}\right)|}{\varsigma}\right) \mho_{\mathscr{I}}(h,\varsigma)  \frac{du}{u}\\
          &= \frac{\mho_{\mathscr{I}}(h,\varsigma)}{\vartheta_{z}\cdot \varsigma} \bigvee\limits_{k\in \mathbb{Z}} |\Phi(z^{n}e^{-k})| n \int\limits_{0}^{\infty} |\psi(u^{n}e^{-k})||\log\left(\frac{u}{z}\right)|  \frac{du}{u}.\\
     \end{align*}
      With the substitution $u^{n}e^{-k} = t$, we obtain 
      \begin{align*}
         E_{2}
         & = \frac{\mho_{\mathscr{I}}(h,\varsigma)}{\vartheta_{z}\cdot \varsigma} \bigvee\limits_{k\in \mathbb{Z}} |\Phi(z^{n}e^{-k})|  \int\limits_{0}^{\infty} |\psi(t)|\left|\log\left(\frac{te^{k}}{z^{n}}\right)^{\frac{1}{n}}\right| \frac{dt}{t}\\
         & \leq \frac{\mho_{\mathscr{I}}(h,\varsigma)}{n  \vartheta_{z}\cdot \varsigma} \bigvee\limits_{k\in \mathbb{Z}} |\Phi(z^{n}e^{-k})|  \int\limits_{0}^{\infty} |\psi(t)|\left\{ |\log{t}| + |k-n\log{z}|\right\}\frac{dt}{t}\\
         & \leq \frac{\mho_{\mathscr{I}}(h,\varsigma)}{n  \vartheta_{z}\cdot \varsigma} \left[ \mathfrak{m}_{0}(\Phi)\cdot \mathrm{M}_{1}(\Psi) +\mathfrak{m}_{1}(\Phi)\cdot \mathrm{M}_{0}(\Psi) \right].
     \end{align*}
     Combining the estimates for $E_{1}$ and $E_{2}$ gives the desired result.
\end{proof}

\section{The Max-Min Durrmeyer type Exponential sampling operators}\label{sec4}
In this section, we examine the pointwise and uniform convergence properties and the rate of approximation of the Max–Min Durrmeyer-type exponential sampling operators.
\begin{definition}\label{Maxmin}
     Let $\mathfrak{h}: \mathbb{R}_{+} \rightarrow [0,1]$ be an $\mathscr{L}^1$-integrable and bounded function on $\mathbb{R}_{+}$. Then, the Max-Min kind Durrmeyer-type exponential sampling operators for $\mathfrak{h}$ with respect to $\Phi$ and $\Psi$ are defined as
     \[ \mathscr{D}_{n,\Phi, \Psi}^{m}\,(\mathfrak{h}) (z) := \bigvee\limits_{k\in \mathbb{Z}}   \left[ n \int\limits_{0}^{\infty} \Psi(u^{n}e^{-k}) \mathfrak{h}(u)\frac{du}{u}\right] \wedge \frac{\Phi(z^{n}e^{-k})}{\bigvee\limits_{k\in \mathbb{Z}} \Phi(z^{n}e^{-k}) n \int\limits_{0}^{\infty} \Psi(u^{n}e^{-k}) \frac{du}{u}}. \] 
\end{definition}
It is easy to  observe that $\bigl| \mathscr{D}_{n,\Phi, \Psi}^{m}\,(\mathfrak{h}) (z) \bigr| \leq 1 $. Hence  the operator is well defined.
\begin{lemma}\label{lm1MaxMin} 
 Let $\mathfrak{h}_{1}, \mathfrak{h}_{2} : \mathbb{R}_{+} \to [0,1]$ be two bounded functions. Then, for all $ z\in\mathbb{R}_{+}$, the following properties hold:
     \begin{enumerate}
         \item[(a)] If $\mathfrak{h}_{2}(z) \geq \mathfrak{h}_{1}(z)$, then $$ \mathscr{D}_{n,\Phi, \Psi}^{m}(\mathfrak{h}_{2}) (z) \geq\mathscr{D}_{n,\Phi, \Psi}^{m}(\mathfrak{h}_{1}) (z).$$ 
         \item[(b)]For sufficiently large $n\in\mathbb{N}$, $$\left|\mathscr{D}_{n,\Phi, \Psi}^{m}(\mathfrak{h}_{1}) (z) - \mathscr{D}_{n,\Phi, \Psi}^{m}(\mathfrak{h}_{2}) (z)\right| \leq \mathscr{D}_{n,\Phi, \Psi}^{m}(|\mathfrak{h}_{1} - \mathfrak{h}_{2}|;z).$$
         \item[(c)]For all $z \in \mathbb{R}_{+}$, $$\mathscr{D}_{n,\Phi, \Psi}^{m}\,(\mathfrak{h}_{1} + \mathfrak{h}_{2}) (z)\leq \mathscr{D}_{n,\Phi, \Psi}^{m}(\mathfrak{h}_{1}) (z)+ \mathscr{D}_{n,\Phi, \Psi}^{m}(\mathfrak{h}_{2}) (z).$$    
     \end{enumerate}
\end{lemma}
\begin{proof}
   From the definition of operators $ \mathscr{D}_{n,\Phi, \Psi}^{m}$ and Lemma~\ref{lm7}, (a) and (c) are obvious. For (c) we use the inequality $ \mathfrak{h}_{1}(s)-\mathfrak{h}_{2}(s)\leq \bigl|\mathfrak{h}_{1}(s)-\mathfrak{h}_{2}(s)\bigr|$ and then interchanging the role of $\mathfrak{h}_{1} \,\, \And\,\,  \mathfrak{h}_{2}$ we get the desired inequality. 
\end{proof}

\begin{theorem}\label{Kantthm1}
Let $\mathfrak{h}:I \to [0,1]$ be a bounded and $\mathscr{L}^1$-integrable function. Then, \[ \lim\limits_{n\to \infty} \mathscr{D}_{n,\Phi, \Psi}^{m}\,(\mathfrak{h}) (z) = \mathfrak{h}(z) \] holds at every point $z\in\mathbb{R}_{+},$ where $\mathfrak{h}$ is log-continuous.

Furthermore, if $\mathfrak{h} \in \mathscr{LU}_{b}(I)$, then
\[\lim\limits_{n\to \infty} \left\|\mathscr{D}_{n,\Phi, \Psi}^{m}\,(\mathfrak{h}) (z) - \mathfrak{h}(z)\right\|_{\infty} = 0, \] that is, the convergence is uniform.
\end{theorem}
\begin{proof}
   Let $\mathfrak{h}:I \to [0,1]\,$ be log-continuous at $z\in \mathbb{R}_{+}$. By the triangle inequality, we have
    \begin{align*}
      \left|\right.&\left.\mathscr{D}_{n,\Phi, \Psi}^{m}\,(\mathfrak{h}) (z) - \mathfrak{h}(z)\right|\\  & = \left| \bigvee\limits_{k\in \mathbb{Z}}   \left[ n \int\limits_{0}^{\infty} \Psi(u^{n}e^{-k}) \mathfrak{h}(u)\frac{du}{u}\right] \wedge \frac{\Phi(z^{n}e^{-k})}{\bigvee\limits_{k\in \mathbb{Z}} \Phi(z^{n}e^{-k}) n \int\limits_{0}^{\infty} \Psi(u^{n}e^{-k})\frac{du}{u}}  - \mathfrak{h}(z)\right|\\
        &\leq  \left| \bigvee\limits_{k\in \mathbb{Z}}   \left[ n \int\limits_{0}^{\infty} \Psi(u^{n}e^{-k}) \mathfrak{h}(u)\frac{du}{u}\right] \wedge \frac{\Phi(z^{n}e^{-k})}{\bigvee\limits_{k\in \mathbb{Z}} \Phi(z^{n}e^{-k}) n \int\limits_{0}^{\infty} \Psi(u^{n}e^{-k}) \frac{du}{u}} \right. \\ & \left. \quad -\bigvee\limits_{k\in \mathbb{Z}}   \left[ n \int\limits_{0}^{\infty} \Psi(u^{n}e^{-k}) \mathfrak{h}(z)\frac{du}{u}\right] \wedge \frac{\Phi(z^{n}e^{-k})}{\bigvee\limits_{k\in \mathbb{Z}} \Phi(z^{n}e^{-k}) n \int\limits_{0}^{\infty} \Psi(u^{n}e^{-k})\frac{du}{u}}\right|\\ &\quad+ \left|\bigvee\limits_{k\in \mathbb{Z}}   \left[ n \int\limits_{0}^{\infty} \Psi(u^{n}e^{-k}) \mathfrak{h}(z)\frac{du}{u}\right] \wedge \frac{\Phi(z^{n}e^{-k})}{\bigvee\limits_{k\in \mathbb{Z}} \Phi(z^{n}e^{-k}) n \int\limits_{0}^{\infty} \Psi(u^{n}e^{-k}) \frac{du}{u}} - \mathfrak{h}(z)\right|\\ 
        & = \quad\mathsf{E}_{1} + \mathsf{E}_{2}.
    \end{align*}
     Estimate of $\mathsf{E}_{1}$.  Using Lemmas \ref{lm5}  and \ref{lm6}, we obtain
   \begin{align*}
       \mathsf{E}_{1} & = \left| \bigvee\limits_{k\in \mathbb{Z}}   \left[ n \int\limits_{0}^{\infty} \Psi(u^{n}e^{-k}) \mathfrak{h}(u)\frac{du}{u}\right] \wedge \frac{\Phi(z^{n}e^{-k})}{\bigvee\limits_{k\in \mathbb{Z}} \Phi(z^{n}e^{-k}) n \int\limits_{0}^{\infty} \Psi(u^{n}e^{-k}) \frac{du}{u}} \right. \\ & \left. \quad -\bigvee\limits_{k\in \mathbb{Z}}   \left[ n \int\limits_{0}^{\infty} \Psi(u^{n}e^{-k}) \mathfrak{h}(z)\frac{du}{u}\right] \wedge \frac{\Phi(z^{n}e^{-k})}{\bigvee\limits_{k\in \mathbb{Z}} \Phi(z^{n}e^{-k}) n \int\limits_{0}^{\infty} \Psi(u^{n}e^{-k})\, \frac{du}{u}}\right|\\ 
       &  \leq   \left|\bigvee\limits_{k\in \mathbb{Z}} \left[n \int\limits_{0}^{\infty} \Psi(u^{n}e^{-k}) (\mathfrak{h}(u) - \mathfrak{h}(z))\frac{du}{u}\right]\right|  \wedge \frac{\Phi(z^{n}e^{-k})}{\bigvee\limits_{k\in \mathbb{Z}} \Phi(z^{n}e^{-k}) n \int\limits_{0}^{\infty} \Psi(u^{n}e^{-k}) \frac{du}{u}}\\
        &  \leq  \bigvee\limits_{k\in \mathbb{Z}}   \left[n \int\limits_{0}^{\infty} \Psi(u^{n}e^{-k}) \left|\mathfrak{h}(u) - \mathfrak{h}(z)\right|\frac{du}{u}\right]  \wedge \frac{\Phi(z^{n}e^{-k})}{\bigvee\limits_{k\in \mathbb{Z}} \Phi(z^{n}e^{-k}) n \int\limits_{0}^{\infty} \Psi(u^{n}e^{-k}) \frac{du}{u}}.
     \end{align*}
     Partitioning the Maximum operation, we get
     \begin{align*}
         \mathsf{E}_{1} & \leq \bigvee\limits_{k\in\mathscr{B}_{\tau, n}}  \left[n \int\limits_{0}^{\infty} \Psi(u^{n}e^{-k}) \left|\mathfrak{h}(u) - \mathfrak{h}(z)\right|\frac{du}{u}\right]  \wedge \frac{\Phi(z^{n}e^{-k})}{\bigvee\limits_{k\in \mathbb{Z}} \Phi(z^{n}e^{-k})n \int\limits_{0}^{\infty} \Psi(u^{n}e^{-k}) \frac{du}{u}} \\ &  \bigvee \bigvee\limits_{k\notin \mathscr{B}_{\tau, n}}  \left[n \int\limits_{0}^{\infty} \Psi(u^{n}e^{-k}) \left|\mathfrak{h}(u) - \mathfrak{h}(z)\right|\frac{du}{u}\right]  \wedge \frac{\Phi(z^{n}e^{-k})}{\bigvee\limits_{k\in \mathbb{Z}} \Phi(z^{n}e^{-k}) n \int\limits_{0}^{\infty} \Psi(u^{n}e^{-k}) \frac{du}{u}} \\
         & \leq \mathsf{E}_{1.1} \bigvee  \mathsf{E}_{1.2}.
     \end{align*}
     For every $\varepsilon > 0$, by the log-continuity of $\mathfrak{h}$ at $z$, we have
    \begin{align*}
        \mathsf{E}_{1.1} &  \leq\bigvee\limits_{k\in\mathscr{B}_{\tau, n}}  \left[n \int\limits_{0}^{\infty} \Psi(u^{n}e^{-k}) \left|\mathfrak{h}(u) - \mathfrak{h}(z)\right|\frac{du}{u}\right]  \wedge \frac{\Phi(z^{n}e^{-k})}{\bigvee\limits_{k\in \mathbb{Z}} \Phi(z^{n}e^{-k}) n \int\limits_{0}^{\infty} \Psi(u^{n}e^{-k}) \frac{du}{u}}\\
       & \leq  \bigvee\limits_{k\in\mathscr{B}_{\tau, n}}  \left[n \int\limits_{0}^{\infty} \Psi(u^{n}e^{-k}) \varepsilon \frac{du}{u}\right]  \wedge \frac{\Phi(z^{n}e^{-k})}{\bigvee\limits_{k\in \mathbb{Z}} \Phi(z^{n}e^{-k}) n \int\limits_{0}^{\infty} \Psi(u^{n}e^{-k}) \frac{du}{u}}\\
      &  \leq  \bigvee\limits_{k\in\mathscr{B}_{\tau, n}}  \varepsilon  \wedge \frac{\Phi(z^{n}e^{-k})}{\bigvee\limits_{k\in \mathbb{Z}} \Phi(z^{n}e^{-k})\quad n \int\limits_{0}^{\infty} \Psi(u^{n}e^{-k})\, \frac{du}{u}}
        \leq \bigvee\limits_{k\in\mathscr{B}_{\tau, n}} \epsilon \land 1
        \leq  \epsilon.   
    \end{align*}
     For $\mathsf{E}_{1.2}$, using $\mathfrak{p} \land \mathfrak{q} \leq \mathfrak{q}$ for all $\mathfrak{p}, \mathfrak{q} \in \mathbb{R}{+}$, we obtain
     \begin{align*}
         \mathsf{E}_{1.2} &  =\bigvee\limits_{k\notin \mathscr{B}_{\tau, n}}  \left[n \int\limits_{0}^{\infty} \Psi(u^{n}e^{-k}) \left|\mathfrak{h}(u) - \mathfrak{h}(z)\right|\frac{du}{u}\right]  \wedge \frac{\Phi(z^{n}e^{-k})}{\bigvee\limits_{k\in \mathbb{Z}} \Phi(z^{n}e^{-k}) n \int\limits_{0}^{\infty} \Psi(u^{n}e^{-k}) \frac{du}{u}}\\
         &\leq  \bigvee\limits_{k\notin \mathscr{B}_{\tau, n}}  \frac{\Phi(z^{n}e^{-k})}{\bigvee\limits_{k\in \mathbb{Z}} \Phi(z^{n}e^{-k}) n \int\limits_{0}^{\infty} \Psi(u^{n}e^{-k}) \frac{du}{u}}
         \leq  \frac{1}{\vartheta_{u}} \bigvee\limits_{k\notin \mathscr{B}_{\tau, n}}\Phi(e^{-k}z^{n})
         \leq  \frac{c\,n^{-\nu} }{\vartheta_{u}}
         \leq  \epsilon,
     \end{align*}
     for sufficiently large $n \in \mathbb{N}$, where $\nu$ is as defined in Lemma \ref{lma2}.
     
    Estimation of $\mathsf{E}_{2}$, 
    \begin{align*}
        \mathsf{E}_{2} & = \left|\bigvee\limits_{k\in \mathbb{Z}}   \left[ n \int\limits_{0}^{\infty} \Psi(u^{n}e^{-k}) \mathfrak{h}(z)\frac{du}{u}\right] \wedge \frac{\Phi(z^{n}e^{-k})}{\bigvee\limits_{k\in \mathbb{Z}} \Phi(z^{n}e^{-k}) n \int\limits_{0}^{\infty} \Psi(u^{n}e^{-k}) \frac{du}{u}} - \mathfrak{h}(z)\right|. 
    \end{align*}
Setting $u^{n}e^{-k}=t$, we have
    \begin{align*}
       \mathsf{E}_{2} &  =\left|\bigvee\limits_{k\in \mathbb{Z}} \quad  \mathfrak{h}(z)\wedge \frac{\Phi(z^{n}e^{-k})}{\bigvee\limits_{k\in \mathbb{Z}} \Phi(z^{n}e^{-k}) n \int\limits_{0}^{\infty} \Psi(u^{n}e^{-k}) \frac{du}{u}} - \mathfrak{h}(z)\right| \\
        & = \left|\bigvee\limits_{k\in \mathbb{Z}} (\mathfrak{h}(z)\wedge 1)\,\wedge \frac{\Phi(z^{n}e^{-k})}{\bigvee\limits_{k\in \mathbb{Z}} \Phi(z^{n}e^{-k}) n \int\limits_{0}^{\infty} \Psi(u^{n}e^{-k}) \frac{du}{u}} -  \mathfrak{h}(z)\right|\\ 
       &  \leq \left| \mathfrak{h}(z)\wedge \bigvee\limits_{k\in \mathbb{Z}} 1 \wedge \frac{\Phi(z^{n}e^{-k})}{\bigvee\limits_{k\in \mathbb{Z}} \Phi(z^{n}e^{-k}) n \int\limits_{0}^{\infty} \Psi(u^{n}e^{-k}) \frac{du}{u}} -  \mathfrak{h}(z)\right|\\ 
       &  \leq \left| \mathfrak{h}(z)\wedge \bigvee\limits_{k\in \mathbb{Z}} \frac{\Phi(z^{n}e^{-k})}{\bigvee\limits_{k\in \mathbb{Z}} \Phi(z^{n}e^{-k}) n \int\limits_{0}^{\infty} \Psi(u^{n}e^{-k}) \frac{du}{u}} -  \mathfrak{h}(z)\right|\\ 
       &  \leq \left| \mathfrak{h}(z) \land 1  - \mathfrak{h}(z)\right| = 0.
    \end{align*}
    This completes the proof.
\end{proof}

\begin{theorem}\label{thm2}
    Let $\mathfrak{h}:I \to [0,1]$, and suppose both $\varsigma_{n}$ and $\frac{1}{n\varsigma_{n}}$ tend to $0$ as $n \to \infty$ with $\nu >0$ satisfying Lemma \ref{lma2}. Then, for any $\mathfrak{h} \in \mathscr{LU}_{b}(I)$, the following inequality holds: 
    \[ \|\mathscr{D}_{n,\Phi, \Psi}^{m}(\mathfrak{h})(z)- \mathfrak{h}(z)\| \leq \mho_{\mathbb{R}_{+}}(\mathfrak{h},\varsigma_{n}) + \left( \mho_{\mathbb{R}_{+}}(\mathfrak{h},\varsigma_{n}) \bigvee \frac{\mathfrak{m}_{\nu}(\Phi)}{\vartheta_{z} n^{\nu}(\varsigma_{n})^{\nu}}\right). \]
\end{theorem}

\begin{proof}
    We begin with
    \begin{align*}
      \|&\mathscr{D}_{n,\Phi, \Psi}^{m}(\mathfrak{h})(z)- \mathfrak{h}(z)\| \\ &  \leq\bigvee\limits_{k\in \mathbb{Z}}  \left[n \int\limits_{0}^{\infty} \Psi(u^{n}e^{-k}) \left|\mathfrak{h}(u) - \mathfrak{h}(z)\right|\frac{du}{u}\right]  \wedge \frac{\Phi(z^{n}e^{-k})}{\bigvee\limits_{k\in \mathbb{Z}} \Phi(z^{n}e^{-k}) n \int\limits_{0}^{\infty} \Psi(u^{n}e^{-k})\frac{du}{u}}.\\
    \end{align*}
    Applying the triangle inequality together with condition $(\Psi.1)$, we obtain
    \begin{align*}
    & \|\mathscr{D}_{n,\Phi, \Psi}^{m}(\mathfrak{h})(z)- \mathfrak{h}(z) \|  \\  &  \leq\bigvee\limits_{k\in \mathbb{Z}} \left[n \int\limits_{0}^{\infty} \Psi(u^{n}e^{-k}) \left|\mathfrak{h}(u) -\mathfrak{h}\left(e^{\frac{k}{n}}\right)\right|\frac{du}{u}\right]  \wedge \frac{\Phi(z^{n}e^{-k})}{\bigvee\limits_{k\in \mathbb{Z}} \Phi(z^{n}e^{-k}) n \int\limits_{0}^{\infty} \Psi(u^{n}e^{-k}) \frac{du}{u}}\\& + \bigvee\limits_{k\in \mathbb{Z}}   \left[n \int\limits_{0}^{\infty} \Psi(u^{n}e^{-k}) \left|\mathfrak{h}\left(e^{\frac{k}{n}}\right)- \mathfrak{h}(z)\right|\frac{du}{u}\right]  \wedge \frac{\Phi(z^{n}e^{-k})}{\bigvee\limits_{k\in \mathbb{Z}} \Phi(z^{n}e^{-k}) n \int\limits_{0}^{\infty} \Psi(u^{n}e^{-k}) \frac{du}{u}}\\
     &  \leq\bigvee\limits_{k\in \mathbb{Z}} \left[n \int\limits_{0}^{\infty} \Psi(u^{n}e^{-k}) \left|\mathfrak{h}(u) -\mathfrak{h}\left(e^{\frac{k}{n}}\right)\right|\frac{du}{u}\right]  \wedge \frac{\Phi(z^{n}e^{-k})}{\bigvee\limits_{k\in \mathbb{Z}} \Phi(z^{n}e^{-k}) n \int\limits_{0}^{\infty} \Psi(u^{n}e^{-k}) \frac{du}{u}}\\& + \bigvee\limits_{k\in \mathbb{Z}}  \left|\mathfrak{h}\left(e^{\frac{k}{n}}\right)- \mathfrak{h}(z)\right| \wedge \frac{\Phi(z^{n}e^{-k})}{\bigvee\limits_{k\in \mathbb{Z}} \Phi(z^{n}e^{-k}) n \int\limits_{0}^{\infty} \Psi(u^{n}e^{-k}) \frac{du}{u}}\\
  & \leq \mathbb{R}_{1} + \mathbb{R}_{2}.
    \end{align*}
    For $\mathbb{R}_{1}$, by applying Definition \ref{LMSDef}, we get
\begin{align*}
      \mathbb{R}_{1} & \leq \bigvee\limits_{k\in \mathbb{Z}}  \left[n \int\limits_{0}^{\infty} \Psi(u^{n}e^{-k}) \mho_{\mathbb{R}_{+}}\left( \mathfrak{h}, \left|u -e^{\frac{k}{n}}\right|\right)\frac{du}{u}\right]  \wedge \frac{\Phi(z^{n}e^{-k})}{\bigvee\limits_{k\in \mathbb{Z}} \Phi(z^{n}e^{-k}) n \int\limits_{0}^{\infty} \Psi(u^{n}e^{-k}) \frac{du}{u}} \\
       & \leq \bigvee\limits_{k\in \mathbb{Z}} \mho_{\mathbb{R}_{+}}\left( \mathfrak{h}, \left|\log{u} -{\frac{k}{n}}\right|\right)  \wedge \frac{\Phi(z^{n}e^{-k})}{\bigvee\limits_{k\in \mathbb{Z}} \Phi(z^{n}e^{-k}) n \int\limits_{0}^{\infty} \Psi(u^{n}e^{-k}) \frac{du}{u}} \\
       & \leq \bigvee\limits_{k\in \mathbb{Z}} \mho_{\mathbb{R}_{+}}\left( \mathfrak{h}, \left|\log{u} -{\frac{k}{n}}\right|\right) \\
       & \leq \mho_{\mathbb{R}_{+}}\left( \mathfrak{h}, \varsigma_{n}\right).
    \end{align*}
Now considering $\mathbb{R}_{2}$,
    \begin{align*}
     \mathbb{R}_{2}& \leq 
       \bigvee\limits_{k\in \mathbb{Z}}    \left|\mathfrak{h}\left(e^{\frac{k}{n}}\right)- \mathfrak{h}(z)\right| \wedge \frac{\Phi(z^{n}e^{-k})}{\bigvee\limits_{k\in \mathbb{Z}} \Phi(z^{n}e^{-k}) n \int\limits_{0}^{\infty} \Psi(u^{n}e^{-k}) \frac{du}{u}}\\
      & \leq 
       \bigvee\limits_{k\in \mathscr{B}_{\tau,n} }    \left|\mathfrak{h}\left(e^{\frac{k}{n}}\right)- \mathfrak{h}(z)\right|\wedge \frac{\Phi(z^{n}e^{-k})}{\bigvee\limits_{k\in \mathbb{Z}} \Phi(z^{n}e^{-k}) n \int\limits_{0}^{\infty} \Psi(u^{n}e^{-k}) \frac{du}{u}}\\& \bigvee \bigvee\limits_{k\notin \mathscr{B}_{\tau,n}}  \left|\mathfrak{h}\left(e^{\frac{k}{n}}\right)- \mathfrak{h}(z)\right| \wedge \frac{\Phi(z^{n}e^{-k})}{\bigvee\limits_{k\in \mathbb{Z}} \Phi(z^{n}e^{-k}) n \int\limits_{0}^{\infty} \Psi(u^{n}e^{-k}) \frac{du}{u}} \\
      &  \leq \max\{R_{2.1},R_{2.2}\}.
    \end{align*}
   For $R_{2.1}$, by the definition of $\mho_{\mathbb{R}_{+}}$, we have 
    \begin{align*}
       R_{2.1} & \leq \bigvee\limits_{k\in \mathscr{B}_{\tau,n} }  \left|\mathfrak{h}\left(e^{\frac{k}{n}}\right)- \mathfrak{h}(z)\right| \wedge \frac{\Phi(z^{n}e^{-k})}{\bigvee\limits_{k\in \mathbb{Z}} \Phi(z^{n}e^{-k}) n \int\limits_{0}^{\infty} \Psi(u^{n}e^{-k}) \frac{du}{u}}\\
       &  \leq\mho_{\mathbb{R}_{+}}\left( \mathfrak{h}, \varsigma_{n}\right).
    \end{align*}
  For $R_{2.2}$, note that $\mathfrak{q}\land\mathfrak{r} \leq \mathfrak{r}$ for all $\mathfrak{q},\mathfrak{r}\in\mathbb{R}{+}$, hence
  \begin{align*}
    \mathbb{R}_{2.2} & \leq\bigvee\limits_{k\notin \mathscr{B}_{\tau,n}}   \left|\mathfrak{h}\left(e^{\frac{k}{n}}\right)- \mathfrak{h}(z)\right| \wedge \frac{\Phi(z^{n}e^{-k})}{\bigvee\limits_{k\in \mathbb{Z}} \Phi(z^{n}e^{-k}) n \int\limits_{0}^{\infty} \Psi(u^{n}e^{-k}) \frac{du}{u}}  \\
    & \leq\bigvee\limits_{k\notin \mathscr{B}_{\tau,n}}   \frac{\Phi(z^{n}e^{-k})}{\bigvee\limits_{k\in \mathbb{Z}} \Phi(z^{n}e^{-k}) n \int\limits_{0}^{\infty} \Psi(u^{n}e^{-k}) \frac{du}{u}} \\
    & \leq \frac{1}{\vartheta_{z}} \bigvee\limits_{\left| n \log{z} -k\right| > n\varsigma_{n}} \Phi(e^{-k}z^{n}).
\end{align*}
    From Definition~\ref{def3}, since $\frac{|n\log z - k|^{\nu}}{n^{\nu}(\varsigma_{n})^{\nu}} > 1$ for $\nu$ as in $(\Delta 3)$, it follows that
    \begin{align*}
     \mathbb{R}_{2.2}&  < \frac{1}{\vartheta_{z} n^{\nu} (\varsigma_{n})^{\nu}} \bigvee\limits_{\left| n \log{z} -k\right| > n\varsigma_{n}} \Phi(e^{-k}z^{n}){\left| n \log{z} -k\right|^{\nu}}\\
       &< \frac{1}{\vartheta_{z} n^{\nu}(\varsigma_{n})^{\nu}} \bigvee\limits_{k \in \mathbb{Z}} \Phi(e^{-k}z^{n}) {\left| n \log{z} -k\right|^{v}}\\
      & = \frac{\mathfrak{m}_{v}(\Phi)}{\vartheta_{z} n^{\nu} (\varsigma_{n})^{\nu}} .
    \end{align*}
    Combining all the above estimates, we finally obtain
  \[ \|\mathscr{D}_{n,\Phi, \Psi}^{m}(\mathfrak{h})(z)- \mathfrak{h}(z)\| \leq \mho_{\mathbb{R}_{+}}(\mathfrak{h},\varsigma_{n}) + \left( \mho_{\mathbb{R}_{+}}(\mathfrak{h},\varsigma_{n}) \bigvee \frac{\mathfrak{m}_{\nu}(\Phi)}{\vartheta_{z} n^{\nu}(\varsigma_{n})^{\nu}}\right), \] 
  which completes the proof.
\end{proof}

\begin{definition}
   Let $\alpha \in (0,1]$. The space of logarithmic Hölder continuous functions of order $\alpha$ is denoted by $\mathcal{L}_{\log}^\alpha$, and is defined as
\[
\mathcal{L}_{\log}^\alpha := \left\{ \mathfrak{h} \in \mathscr{LU}_{b}(\mathscr{\mathbb{R}_{+}}) : \exists ~\lambda > 0 \ \text{such that} \ |\mathfrak{h}(z) - \mathfrak{h}(y)| \leq \lambda  |\log z - \log y|^\alpha \right\},
\]  for all $z,y \in \mathbb{R}_{+}$.
\end{definition} 
In the context of the above framework, when considering functions belonging to the space $\mathcal{L}_{\log}^\alpha$, we obtain the following rate of approximation result.

\begin{theorem}
    Let $\mathfrak{h} \in \mathcal{L}_{\log}^\alpha$. Then
\[ \| \mathscr{D}_{n,\Phi,\Psi}^{m}(\mathfrak{h})(z) - \mathfrak{h}(z) \|_{\infty} = \mathcal{O}\left( n^{-\frac{\alpha}{1+\alpha}} \right) \quad \text{as} \quad n \to \infty.\]
\end{theorem}

\begin{proof} 
From Definition \ref{Maxmin} and Lemma \ref{lm6}, we have
    \begin{align*}
      \|&\mathscr{D}_{n,\Phi, \Psi}^{m}(\mathfrak{h})(z)- \mathfrak{h}(z)\| \\&  \leq \bigvee\limits_{k\in \mathbb{Z}}  \left[n \int\limits_{0}^{\infty} \Psi(u^{n}e^{-k}) \left|\mathfrak{h}(u) - \mathfrak{h}(z)\right|\frac{du}{u}\right]  \wedge \frac{\Phi(z^{n}e^{-k})}{\bigvee\limits_{k\in \mathbb{Z}} \Phi(z^{n}e^{-k}) n \int\limits_{0}^{\infty} \Psi(u^{n}e^{-k}) \frac{du}{u}}.
    \end{align*}
Using the inequality
$$|\mathfrak{h}(u) - \mathfrak{h}(z)| \leq |\mathfrak{h}(u) - \mathfrak{h}(e^{\frac{k}{n}})| + |\mathfrak{h}(e^{\frac{k}{n}}) - \mathfrak{h}(z)|,$$
and applying Lemma \ref{lm7}, we obtain
\begin{align*}
  \left|\right.&\left.\mathscr{D}_{n,\Phi, \Psi}^{m}(\mathfrak{h})(z)- \mathfrak{h}(z) \right| \\& \leq \bigvee\limits_{k\in \mathbb{Z}}  \left[n \int\limits_{0}^{\infty} \Psi(u^{n}e^{-k}) \left|\mathfrak{h}(u) -\mathfrak{h}\left(e^{\frac{k}{n}}\right)\right|\frac{du}{u}\right]  \wedge \frac{\Phi(z^{n}e^{-k})}{\bigvee\limits_{k\in \mathbb{Z}} \Phi(z^{n}e^{-k}) n \int\limits_{0}^{\infty} \Psi(u^{n}e^{-k}) \frac{du}{u}}\\ & + \bigvee\limits_{k\in \mathbb{Z}}  \left[n \int\limits_{0}^{\infty} \Psi(u^{n}e^{-k}) \left|\mathfrak{h}\left(e^{\frac{k}{n}}\right)- \mathfrak{h}(z)\right|\frac{du}{u}\right]  \wedge \frac{\Phi(z^{n}e^{-k})}{\bigvee\limits_{k\in \mathbb{Z}} \Phi(z^{n}e^{-k}) n \int\limits_{0}^{\infty} \Psi(u^{n}e^{-k}) \frac{du}{u}}.
\end{align*}
Since $\mathfrak{h} \in \mathcal{L}_{\log}^\alpha$, it follows that
$$|\mathfrak{h}(u) - \mathfrak{h}(z)| \leq \lambda |\log{u}-\log{z}|^{\alpha}, \quad \text{for some}\quad \lambda>0.$$
Therefore,
\begin{align*}
  \left|\right.&\left. \mathscr{D}_{n,\Phi, \Psi}^{m}(\mathfrak{h})(z)- \mathfrak{h}(z) \right|\\   &\leq  \bigvee\limits_{k\in \mathbb{Z}}  \left[n \int\limits_{0}^{\infty} \Psi(u^{n}e^{-k})\lambda \left|\log{u} -{\frac{k}{n}}\right|^{\alpha}\frac{du}{u}\right]  \wedge \frac{\Phi(z^{n}e^{-k})}{\bigvee\limits_{k\in \mathbb{Z}} \Phi(z^{n}e^{-k}) n \int\limits_{0}^{\infty} \Psi(u^{n}e^{-k}) \frac{du}{u}}\\ & + \bigvee\limits_{k\in \mathbb{Z}}    \left[n \int\limits_{0}^{\infty} \Psi(u^{n}e^{-k})\lambda \left|\log{z} -{\frac{k}{n}}\right|^{\alpha}\frac{du}{u}\right]  \wedge \frac{\Phi(z^{n}e^{-k})}{\bigvee\limits_{k\in \mathbb{Z}} \Phi(z^{n}e^{-k}) n \int\limits_{0}^{\infty} \Psi(u^{n}e^{-k})\frac{du}{u}}\\
     &=  \mathrm{B}_{1} + \mathrm{B}_{2}.
\end{align*}
Estimation of $B_{1}$:\\
Let $t = u^{n}e^{-k}$. Then, for some $\lambda > 0$, we have
\begin{align*}
   \mathrm{B}_{1}&  \leq\quad \bigvee\limits_{k\in \mathbb{Z}}\lambda\int\limits_{0}^{\infty} |\Psi(t)|\left| \frac{\log{t}}{n} \right|^{\alpha}\frac{dt}{t} \wedge \frac{\Phi(z^{n}e^{-k})}{\bigvee\limits_{k\in \mathbb{Z}} \Phi(z^{n}e^{-k}) n \int\limits_{0}^{\infty} \Psi(u^{n}e^{-k}) \frac{du}{u}}\\
    &\leq \quad \frac{\lambda \mathrm{M_{\alpha}(\Psi)}}{n^{\alpha}}.
\end{align*}
Estimation of $B_{2}$:
Let $\varsigma_{n} = n^{-\frac{1}{1+\alpha}}$. Then,
\begin{align*}
 \mathrm{B}_{2} &  =\bigvee\limits_{|k-n\log{z}|\leq n\varsigma_{n} } \left[n \int\limits_{0}^{\infty} \Psi(u^{n}e^{-k})\lambda \left|\log{z} -{\frac{k}{n}}\right|^{\alpha}\,\frac{du}{u}\right] \\&\qquad \wedge \frac{\Phi(z^{n}e^{-k})}{\bigvee\limits_{k\in \mathbb{Z}} \Phi(z^{n}e^{-k}) n \int\limits_{0}^{\infty} \Psi(u^{n}e^{-k})  \frac{du}{u}}\\ & \bigvee \bigvee\limits_{|k-n\log{z}| > n\varsigma_{n}}  \left[n \int\limits_{0}^{\infty} \Psi(u^{n}e^{-k})\lambda \left|\log{z} -{\frac{k}{n}}\right|^{\alpha}\frac{du}{u}\right] \\&\qquad \wedge \frac{\Phi(z^{n}e^{-k})}{\bigvee\limits_{k\in \mathbb{Z}} \Phi(z^{n}e^{-k}) n \int\limits_{0}^{\infty} \Psi(u^{n}e^{-k})\frac{du}{u}}\\
       & =  \max\{B_{2.{1}}, B_{2.2} \}.
\end{align*}
For $B_{2.1}$, using the definition of $\mathcal{B}_{\rho_{n}, n}$, we have
\begin{align*}
   B_{2.1} & =\bigvee\limits_{|k-n\log{z}|\leq n\varsigma_{n} }  \left[n \int\limits_{0}^{\infty} \Psi(u^{n}e^{-k})\lambda \left|\log{z} -{\frac{k}{n}}\right|^{\alpha}\frac{du}{u}\right]\\&\qquad  \wedge \frac{\Phi(z^{n}e^{-k})}{\bigvee\limits_{k\in \mathbb{Z}} \Phi(z^{n}e^{-k}) n \int\limits_{0}^{\infty} \Psi(u^{n}e^{-k})  \frac{du}{u}}\\
  &  \leq \bigvee\limits_{|k-n\log{z}|\leq n\varsigma_{n} }  \left[n \int\limits_{0}^{\infty} \Psi(u^{n}e^{-k})\lambda \varsigma_{n}^{\alpha}\frac{du}{u}\right]  
   \leq  \frac{\lambda}{n^{\frac{\alpha}{1+ \alpha}}}.
\end{align*}
For $B_{2.2}$, we obtain
\begin{align*}
    B_{2.2} & =  \bigvee\limits_{|k-n\log{z}| > n\varsigma_{n}}  \left[n \int\limits_{0}^{\infty} \Psi(u^{n}e^{-k})\lambda \left|\log{z} -{\frac{k}{n}}\right|^{\alpha}\frac{du}{u}\right] \\&\qquad \wedge \frac{\Phi(z^{n}e^{-k})}{\bigvee\limits_{k\in \mathbb{Z}} \Phi(z^{n}e^{-k}) n \int\limits_{0}^{\infty} \Psi(u^{n}e^{-k}) \frac{du}{u}}\\
    & \leq \bigvee\limits_{|k-n\log{z}| > n\varsigma_{n}} \frac{\Phi(z^{n}e^{-k})}{\bigvee\limits_{k\in \mathbb{Z}} \Phi(z^{n}e^{-k}) n \int\limits_{0}^{\infty} \Psi(u^{n}e^{-k}) \frac{du}{u}}\\
    & \leq \bigvee\limits_{|k-n\log{z}| > n\varsigma_{n}} \frac{\Phi(e^{-k} z^{n})}{\vartheta_{z}} \frac{\left|n\log{z}- k\right|}{n\varsigma_{n}}
    \leq \frac{\mathfrak{m}_{1}(\Phi)}{\vartheta_{z}}\cdot\frac{1}{n^\frac{\alpha}{1+\alpha}}.
\end{align*}
Combining all the above estimates, we conclude that
\[ \| \mathscr{D}_{n,\Phi,\Psi}^{m}(\mathfrak{h})(z) - \mathfrak{h}(z) \|_{\infty} = \mathcal{O}\left( n^{-\frac{\alpha}{1+\alpha}} \right) \quad \text{as} \quad n \to \infty.\]
 This completes the proof.
\end{proof}

\section{Example Kernels and Numerical Illustration}\label{sec5}

In this section, we introduce two key examples of Mellin-type kernels that serve as the foundational components in constructing the proposed Durrmeyer-type exponential sampling operators.

The first example is the Mellin B-spline kernel, obtained by extending the classical $B$-spline concept to the Mellin transform setting via the logarithmic substitution \( t = \log z \). Formally, for \( z > 0 \) and order \( n \in \mathbb{N} \), the kernel is defined as  
\[
B_n(z) = \frac{1}{(n-1)!} \sum_{k=0}^n (-1)^k \binom{n}{k} \left( \frac{n}{2} + \log z - k \right)_+^{n-1},
\]
where the positive part function \((x)_+\) is given by \(\max\{x, 0\}\). The Mellin $B$-spline kernel is compactly supported on the interval \(\left[e^{-\frac{n}{2}}, e^{\frac{n}{2}}\right]\) and exhibits smoothness and integrability properties essential for the operator analysis. For our purposes, we consider $\Phi$ to be the Mellin $B$-spline kernel of order two, given by 
\[B_{2}(z) = \max \{0,\, 1 - |\log z|\}, \quad z > 0.\]  

The second kernel used is the Mellin-Fejér kernel, parameterized by \(\beta \geq 1\) and \( t \in \mathbb{R} \), and defined as  
\[
F_\beta^t(z) = \frac{\beta}{2\pi z^{t}} \left[ \sinc\left( \frac{\beta \log \sqrt{z}}{\pi} \right) \right]^2,
\]
where \(\sinc(x) = \frac{\sin(\pi x)}{\pi x}\).

In this study, we take $\Psi$ to be the Mellin- Fejér kernel with  \(\beta = \pi\) and \( t = 0\), leading to the simplified form  
\[
F_\pi^0(z) = \frac{1}{2} \left[\sinc\left( \frac{\log z}{2} \right) \right]^2.
\]

The  Mellin $B$-spline kernel  $\Phi$ satisfies the conditions  $(\Phi.1)$ and $(\Phi.2)$, while the Mellin-Fejér kernel $\Psi$ satisfies $(\Psi.1)$  and $(\Psi.2)$. Together, they form an admissible pair for analyzing the convergence behavior of the Durrmeyer-type exponential sampling operators. Within this framework, the approximation and convergence properties of both Max–Product and Max–Min Durrmeyer-type exponential sampling operators can be rigorously investigated.
\subsection*{Numerical Experiments and Graphical Results}

To evaluate the performance and convergence characteristics of the proposed operators, we employ two test functions defined on the interval \([0, 3]\). The first is a smooth oscillatory function, given by
\[
f(s) = \frac{\arctan\left( \frac{\sin(\pi u) + 1}{1+ u^2} \right)}{\arctan(2)},
\]

and the second is a piecewise-defined function given by 

\[
g(s) = \begin{cases}
0.1 + \frac{0.8}{9} u^2, & 0 \leq u < 1.1, \\
0.9 - 0.4 \sin^2\left( 2\pi (u - 1.1) \right), & 1.1 \leq u < 2.0, \\
0.3 + 0.7 (3-u), & 2.0 \leq u \leq 3, \\
0, & \text{otherwise}.
\end{cases}
\]
Both test functions are continuous and bounded, making them suitable for examining the approximation accuracy and stability of the proposed operators.

\subsection*{Approximation of the Smooth Function $f$}

Figures \ref{fig:maxprod_f} and \ref{fig:maxmin_f} illustrate the results of applying the Max-Product and Max-Min Durrmeyer operators to the smooth function $f$. A comparative analysis of their performance for different values of $n$ is presented in Figure \ref{fig:comp_f}. It can be observed that the Durrmeyer-type Max-Product operator demonstrates a slight advantage over its Max–Min counterpart, a trend that is further confirmed by the pointwise error values reported in Tables \ref{t1f} and \ref{t2f}. 
\begin{figure}[H]
    \centering
    \includegraphics[width=0.9\textwidth]{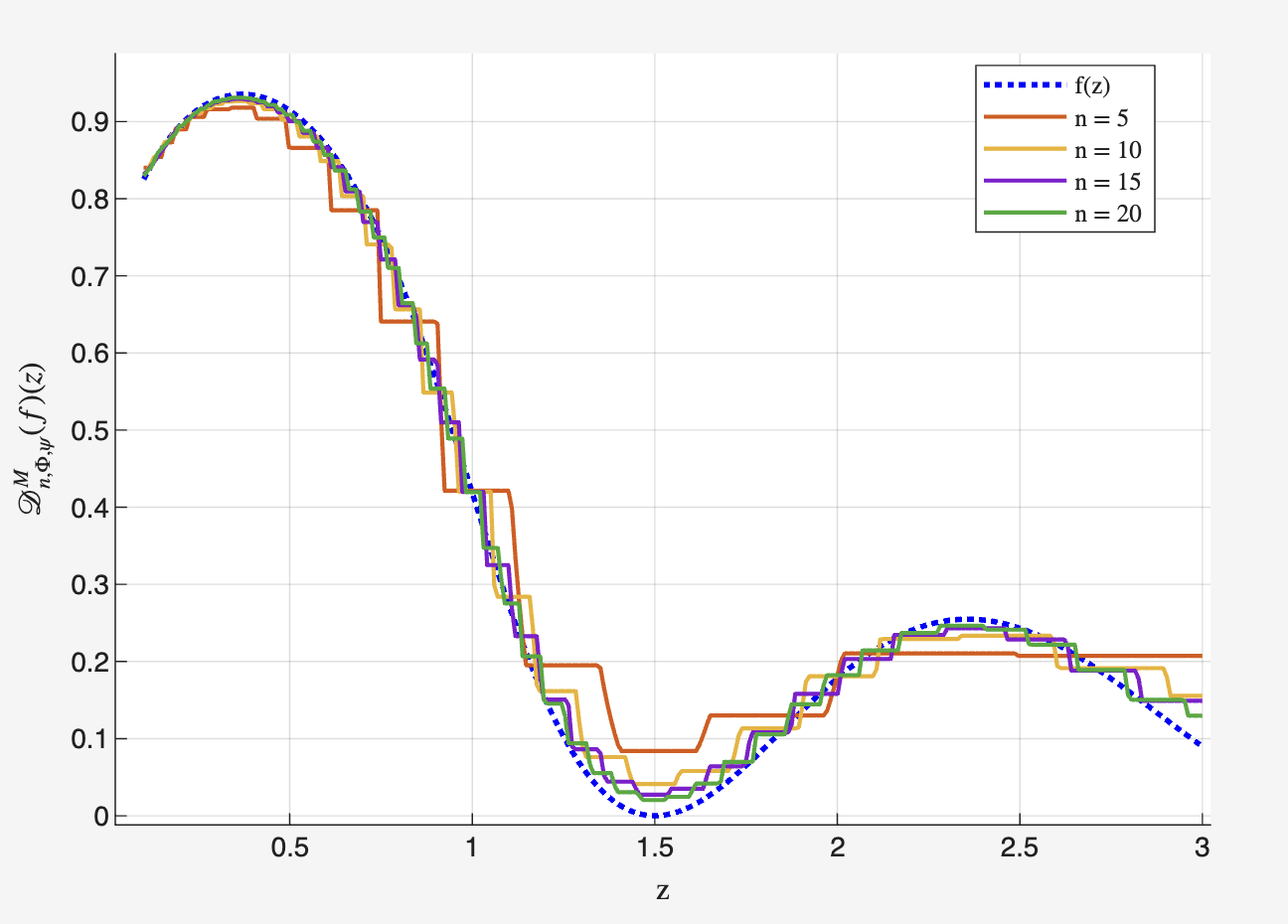}
    \caption{Approximation of $f$ using the Max-Product Durrmeyer operator for different values of $n$.}
    \label{fig:maxprod_f}
\end{figure}
\vspace{0.2em}
\begin{figure}[H]
    \centering
    \includegraphics[width= 0.9\textwidth]{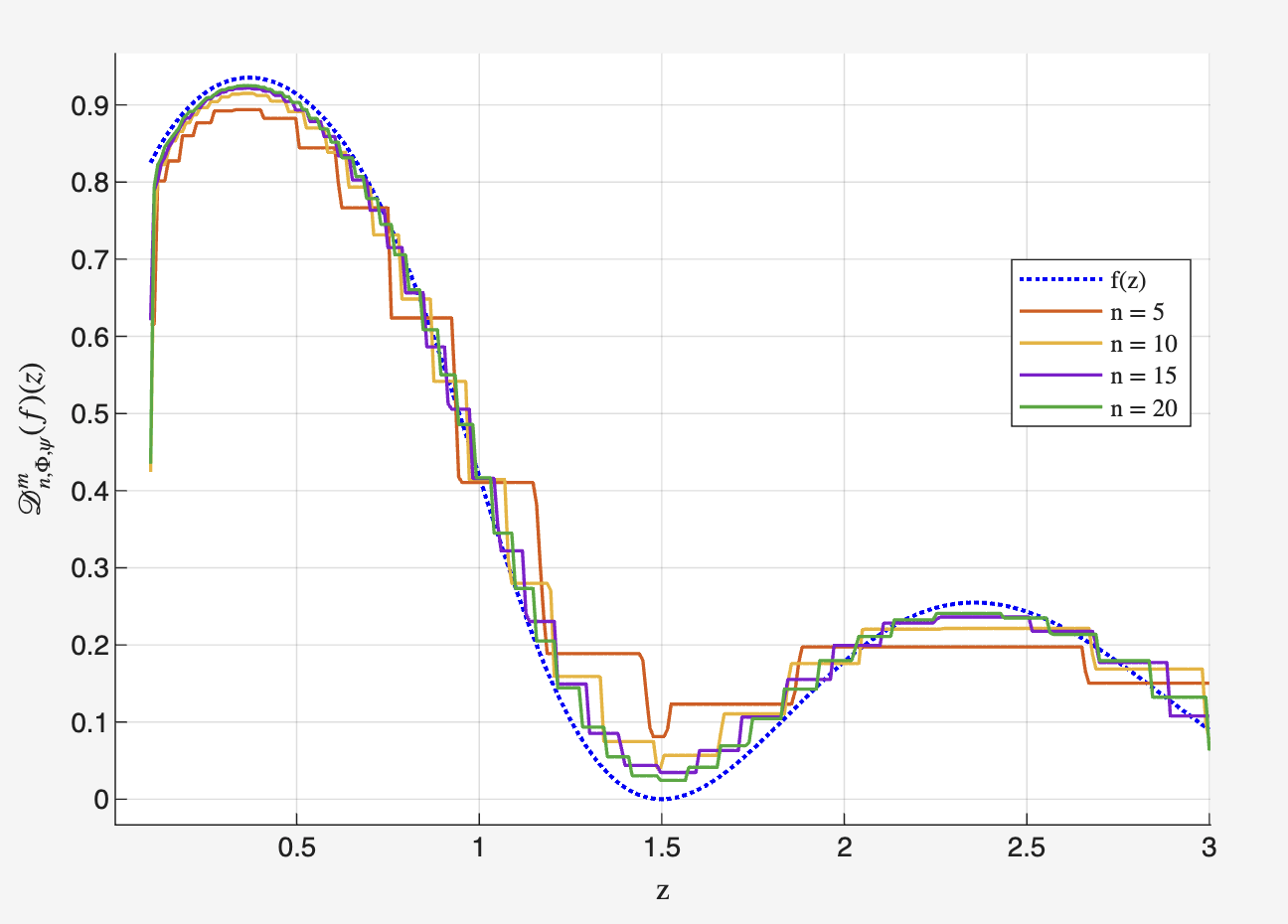}
    \caption{Approximation of $f$ using the Max-Min Durrmeyer operator for different values of $n$.}
    \label{fig:maxmin_f}
\end{figure}
\begin{figure}[H]
    \centering
    \includegraphics[width=0.9\textwidth]{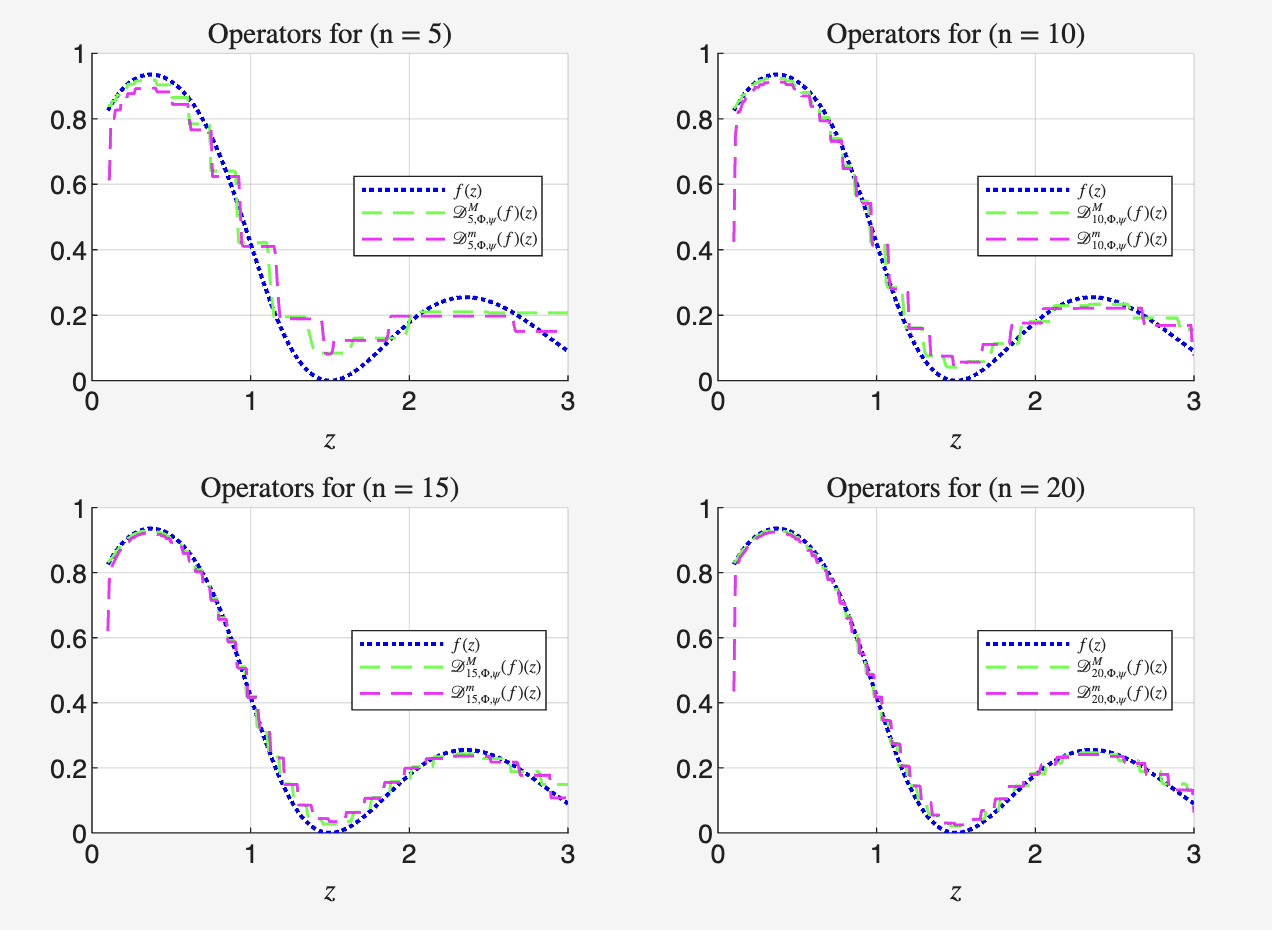}
    \caption{Comparison of the approximations of $f$ by Max-Product and Max-Min Durrmeyer operator.}
    \label{fig:comp_f}
\end{figure}

The corresponding pointwise error tables for each operator are presented below:
\setlength{\tabcolsep}{14pt}
\begin{table}[ht]
\centering
\large
\caption{Pointwise absolute errors of the Max-Product Durrmeyer operator for the function $f$.}\label{t1f}
\begin{tabular}{|c|c|c|c|c|}
\hline
$z$ & $n=5$ & $n=10$ & $n=15$ & $n=20$ \\
\hline
0.3 & 0.013047 & 0.006167  & 0.003983  & 0.002951 \\
0.8 & 0.054074 & 0.038152  & 0.032861  & 0.030167  \\
1.5 & 0.08393  & 0.041165  & 0.027289  & 0.020454  \\
2.2 & 0.029278 & 0.010613  & 0.005212  & 0.002913 \\
2.8 & 0.046903 & 0.030925  & 0.027645  & 0.010275  \\
\hline
\end{tabular}
\end{table}
\setlength{\tabcolsep}{14pt}
\begin{table}[ht]
\centering
\large
\caption{Pointwise absolute errors of the Max-Min Durrmeyer operator for the function $f$.}\label{t2f}
\begin{tabular}{|c|c|c|c|c|}
\hline
$z$ & $n=5$ & $n=10$ & $n=15$ & $n=20$ \\
\hline
0.3 & 0.036776 & 0.018713 & 0.012502 & 0.009253 \\
0.8 & 0.070915 & 0.046449 & 0.038353 & 0.01103   \\
1.5 & 0.081093 & 0.056877 & 0.034616 & 0.024386  \\
2.2 & 0.042393 & 0.019329 & 0.011395 & 0.007278 \\
2.8 & 0.010219 & 0.00813  & 0.016618  & 0.019043 \\
\hline
\end{tabular}
\end{table}

\subsection*{Approximation of the Piecewise Function $g$}

The performance of the operators on the piecewise test function $g$ is illustrated in Figures \ref{fig:maxprod_g}–\ref{fig:comp_g}. The corresponding estimation errors are presented in the pointwise error Tables \ref{t1g} and \ref{t2g}, which further demonstrate the convergence behavior of the proposed operators.
\begin{figure}[H]
    \centering
    \includegraphics[width=0.9\textwidth]{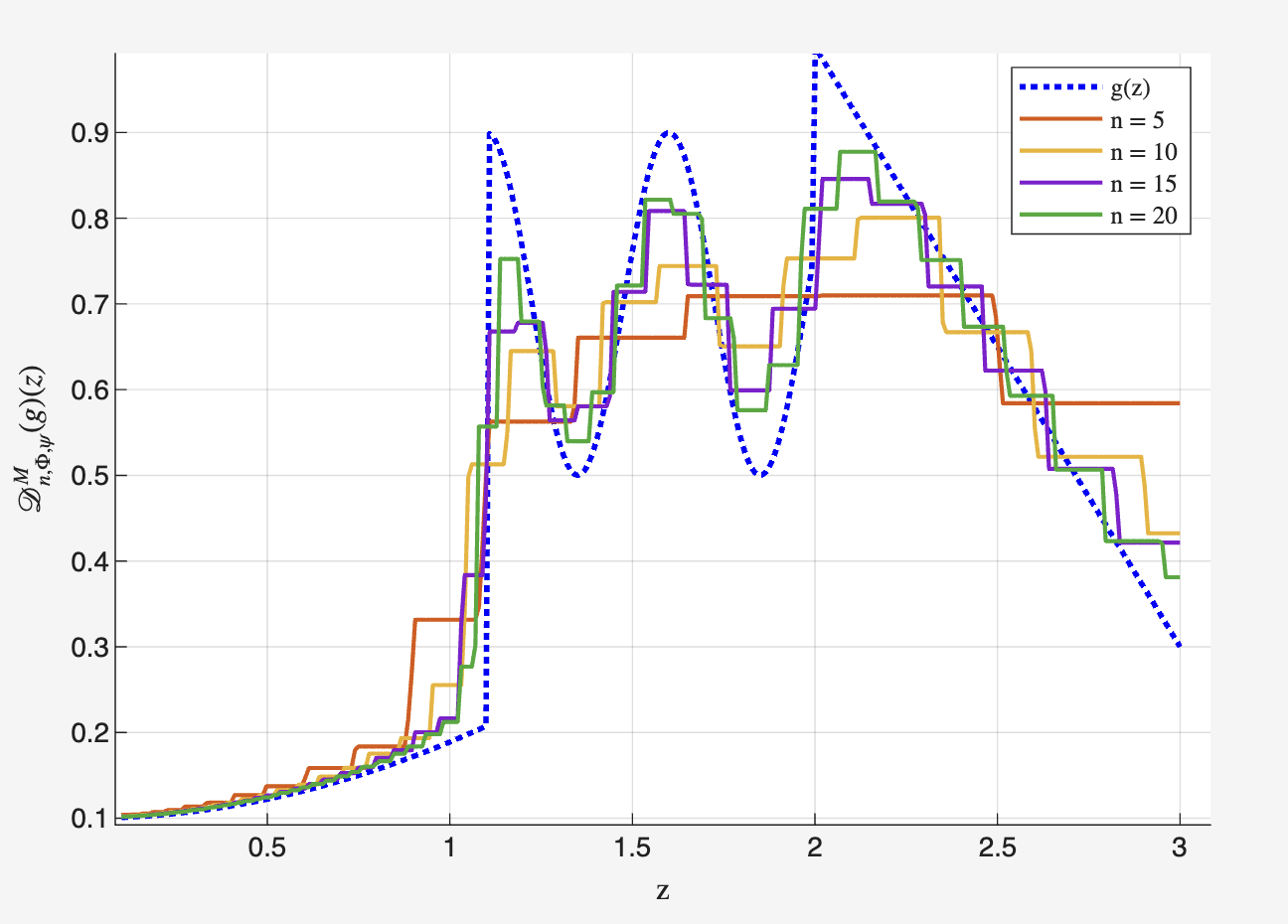}
    \caption{Approximation of $g$ using the Max-Product Durrmeyer operator for different values $n$.}
    \label{fig:maxprod_g}
\end{figure}
\begin{figure}[H]
    \centering
    \includegraphics[width=0.9\textwidth]{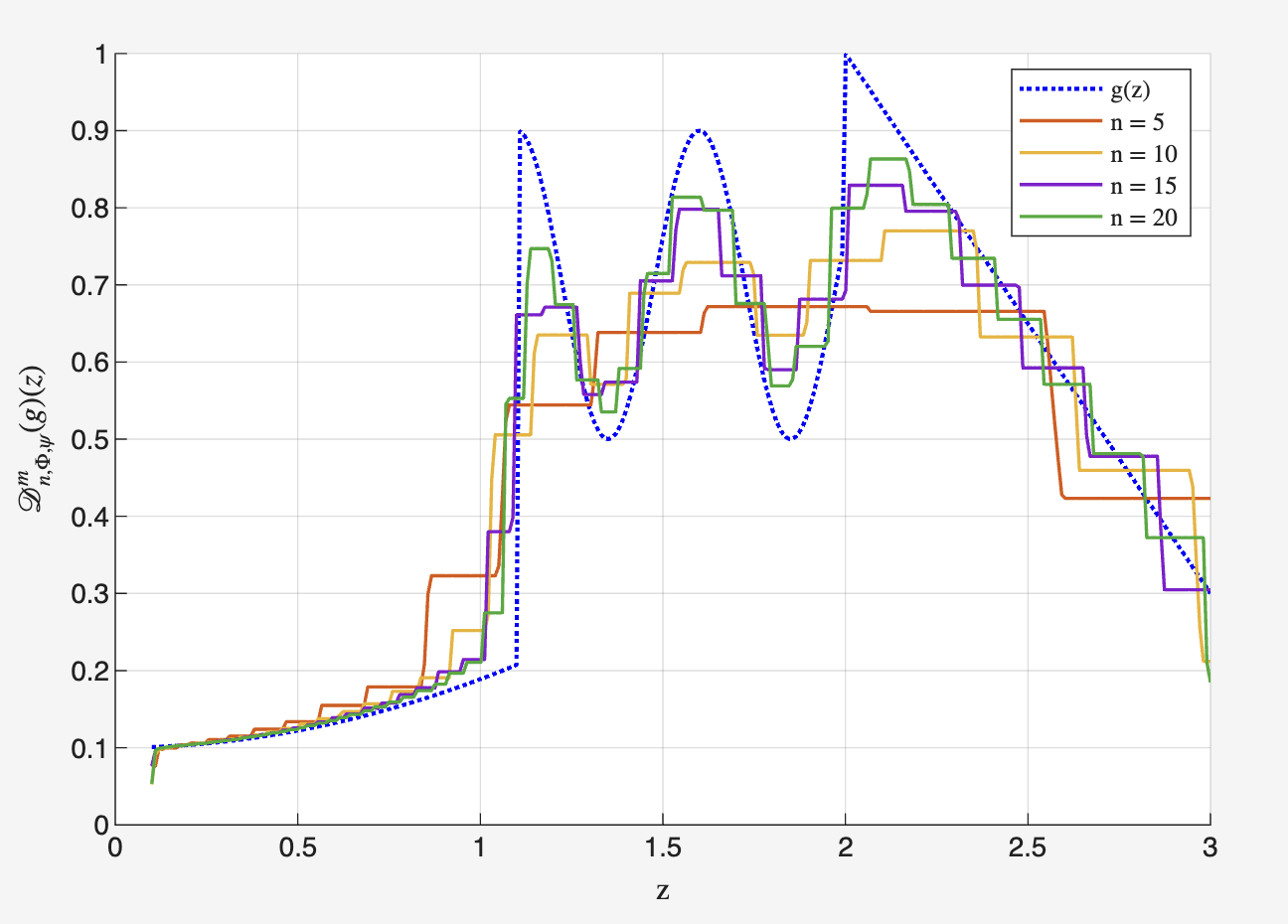}
    \caption{Approximation of $g$ using the Max-Min Durrmeyer operator for different values of $n$.}
    \label{fig:maxmin_g}
\end{figure}
\begin{figure}[H]
    \centering
    \includegraphics[width=0.9\textwidth]{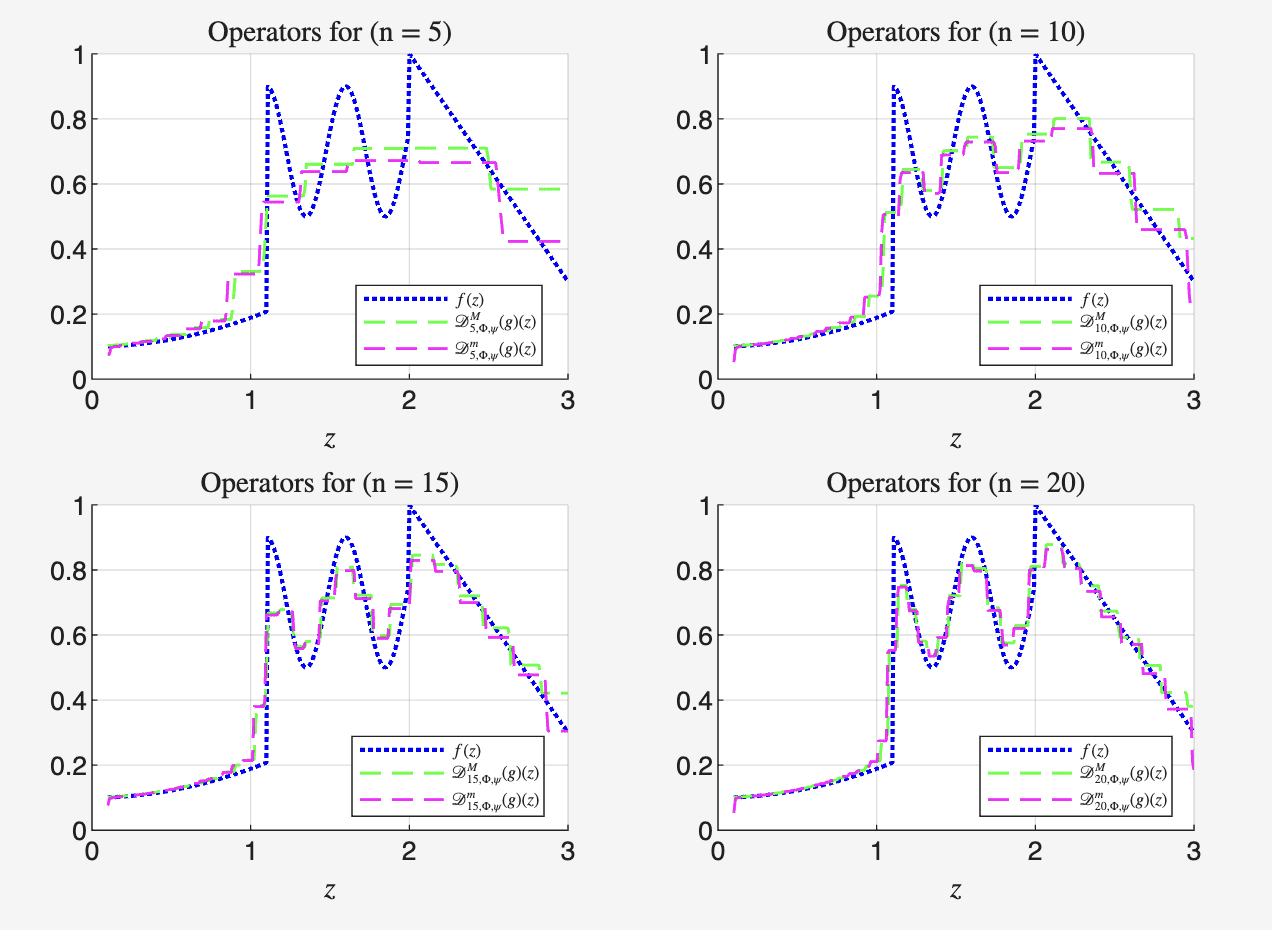}
    \caption{Comparison of the approximations of $g$ by Max-Product and Max-Min Durrmeyer operator.}
    \label{fig:comp_g}
\end{figure}

The corresponding pointwise error tables for each operator are presented below:
\setlength{\tabcolsep}{12pt}
\begin{table}[ht]
\centering
\large
\caption{Pointwise absolute errors of the Max-Product Durrmeyer operator for the function $g$.}\label{t1g}
\begin{tabular}{|c|c|c|c|c|}
\hline
$z$ & $n=5$ & $n=10$ & $n=15$ & $n=20$ \\
\hline
0.3 & 0.00548  & 0.00270  & 0.00171  & 0.00132 \\
0.8 & 0.02673  & 0.01842  & 0.01344  & 0.00967 \\
1.5 & 0.10118  & 0.05965  & 0.04769  & 0.04039 \\
2.2 & 0.15006  & 0.05942  & 0.04307  & 0.04060 \\
2.8 & 0.14416  & 0.08171  & 0.06752  & 0.01667 \\
\hline
\end{tabular}
\end{table}
\setlength{\tabcolsep}{12pt}
\begin{table}[ht]
\centering
\large
\caption{Pointwise absolute errors of the Max-Min Durrmeyer operator for the function $g$.}\label{t2g}
\begin{tabular}{|c|c|c|c|c|}
\hline
$z$ & $n=5$ & $n=10$ & $n=15$ & $n=20$ \\
\hline
0.3 & 0.00254  & 0.00119  & 0.00070   & 0.00058 \\
0.8 & 0.02191  & 0.01620  & 0.01203   & 0.00862  \\
1.5 & 0.12352  & 0.07252  & 0.05660   & 0.04696   \\
2.2 & 0.19429  & 0.08987  & 0.06406   & 0.05570   \\
2.8 & 0.01671  & 0.01959  & 0.03778   & 0.04121   \\
\hline
\end{tabular}
\end{table}



\section{Concluding Remarks:}
In summary, this work introduced the Durrmeyer variants of the Max-Product and Max-Min exponential sampling operators. Their well-definedness, along with pointwise and uniform convergence properties, was rigorously analyzed. The approximation behavior of these operators was further investigated using the logarithmic modulus of continuity. To substantiate the theoretical findings, graphical illustrations and pointwise error tables were presented for two representative test functions, employing the Mellin $B$-spline and Mellin–Fejér kernels. Additional comparative plots highlighted the convergence behavior of both operators. The results reveal that, as $n$ increases, both operators effectively approximate the target functions, with the Max-Product Durrmeyer exponential sampling operator exhibiting slightly faster convergence than its Max-Min counterpart. Many applications, particularly in signal and image processing, have been developed using these operators. Several researchers—such as those in \cite{bajpeyi2022approximation, Acar2025, bardaro2021durrmeyer, cai2024convergence} have recently made significant contributions to the literature on sampling type operators. They have established Durrmeyer-type exponential sampling operators and derived both pointwise and uniform convergence theorems, along with quantitative estimates for the corresponding order of approximation.

\bmhead{Acknowledgements}
The authors would like to thank to the reviewers for their comments and suggestions to improve the article in this format.

\section*{Declarations}

\bmhead{Funding}This work is not funded by any agencies.
\bmhead{Conflict of interest}There is no conflict of interest regarding the publication of this article.
\bmhead{Data availability}Not applicable.



\end{document}